\documentclass[12pt,reqno]{amsart}
\pdfoutput=1
\usepackage[table]{xcolor}
\usepackage{colortbl}
\usepackage{booktabs}
\usepackage{multirow}
\usepackage{subfig}
\usepackage[headings]{fullpage}
\usepackage{amssymb,bbm}
\usepackage{epstopdf}
\usepackage{graphicx}
\usepackage{texdraw}
\usepackage{url}
\usepackage[all]{xy}
\usepackage[T1]{fontenc}
\usepackage{mathtools}
\usepackage{float}   
\usepackage[bookmarks=true,%
    colorlinks=true,%
    linkcolor=blue,%
    citecolor=blue,%
    filecolor=blue,%
    menucolor=blue,%
    urlcolor=blue,%
    breaklinks=true]{hyperref}
\usepackage{stmaryrd}

\newtheorem{theorem}{Theorem}

\newtheorem{observation}[theorem]{Observation}
\theoremstyle{definition}
\newtheorem{proposition}[theorem]{Proposition}
\newtheorem{lemma}[theorem]{Lemma}

\newtheorem{conjecture}[theorem]{Conjecture}

\def\BZ{\mathbbm Z}
\def\BQ{\mathbbm Q}
\def\BR{\mathbbm R}
\def\BC{\mathbbm C}

\def\s{\sigma}

\def\PSL{\mathrm{PSL}}
\def\GL{\mathrm{GL}}

\def\tq{\tilde{q}}

\def\Vol{\mathrm{Vol}}
\def\a{\alpha} 
\def\z{\zeta}

\def\={\;=\;}
\def\be{\begin{equation}}
\def\ee{\end{equation}}

\def\Li{\mathrm{Li}}
\def\e{\bold e}

\def\th{\theta}

\def\ve{\varepsilon}

\def\SU{\mathrm{SU}}
\def\SO{\mathrm{SO}}
\def\Ahat{\widehat{A}}

\def\Im{\mathrm{Im}}

\def\diag{\mathrm{diag}}

\def\e{\bold e}  
\def\TV{\mathrm{TV}}
\def\ev{\mathrm{ev}}
\def\Irot{I^{\mathrm{rot}}}
\def\Ifug{I^{\sharp}}
\def\Imer{I^\mathrm{mer}}
\def\b{\beta}

\def\Ahat{\hat{A}}

\def\GL{\mathrm{GL}}

\def\sfb{\mathsf{b}}
\def\hb{h}

\newcommand{\wh}[1]{\widehat{#1}}

\graphicspath{{figs/}}

\makeatletter

\renewcommand\thepart{\@Roman\c@part}%
\renewcommand\part{%
   \if@noskipsec \leavevmode \fi
   \par
   \addvspace{6.7ex}%
   \@afterindentfalse
   \secdef\@part\@spart}
\def\@part[#1]#2{%
    \ifnum \c@secnumdepth >\m@ne
      \refstepcounter{part}%
      \addcontentsline{toc}{part}{Part~\thepart.\ #1}%
    \else
      \addcontentsline{toc}{part}{#1}%
    \fi
    {\parindent \z@ \raggedright
     \interlinepenalty \@M
     \normalfont
     \ifnum \c@secnumdepth >\m@ne
       \centering\large\scshape \partname~\thepart.%
       \hspace{1ex}%
     \fi%
     \large\scshape #2%
     \markboth{}{}\par}%
    \nobreak
    \vskip 4.7ex
    \@afterheading}
  \def\@spart#1{
  \refstepcounter{part}%
  \addcontentsline{toc}{part}{#1}%
    {\parindent \z@ \raggedright
     \interlinepenalty \@M
     \normalfont
     \centering\large\scshape #1\par}%
     \nobreak
     \vskip 4.7ex
     \@afterheading}
\renewcommand*\l@part[2]{%
  \ifnum \c@tocdepth >-2\relax
    \addpenalty\@secpenalty
    \addvspace{0.75em \@plus\p@}%
    \begingroup
      \parindent \z@ \rightskip \@pnumwidth
      \parfillskip -\@pnumwidth
      {\leavevmode
       \normalsize \bfseries #1\hfil \hb@xt@\@pnumwidth{\hss #2}}\par
       \nobreak
       \if@compatibility
         \global\@nobreaktrue
         \everypar{\global\@nobreakfalse\everypar{}}%
      \fi
    \endgroup
  \fi}

\def\l@subsection{\@tocline{2}{0pt}{2pc}{6pc}{}}
\makeatother

\begin{document}
\title[Periods, the meromorphic 3D-index and the Turaev--Viro invariant]{
  Periods, the meromorphic 3D-index and the Turaev--Viro invariant}
\author{Stavros Garoufalidis}
\address{
  International Center for Mathematics, Department of Mathematics \\
  Southern University of Science and Technology \\
  Shenzhen 518055, China \newline
  {\tt \url{http://people.mpim-bonn.mpg.de/stavros}}}
\email{stavros@mpim-bonn.mpg.de}


\author{Campbell Wheeler}
\address{Max Planck Institute for Mathematics \\
         Vivatsgasse 7, 53111 Bonn, Germany \newline
         {\tt \url{http://guests.mpim-bonn.mpg.de/cjwh}}}
\email{cjwh@mpim-bonn.mpg.de}


\date{3 September 2022}

\begin{abstract}
  The 3D-index of Dimofte--Gaiotto--Gukov is an interesting collection of $q$-series
  with integer coefficients parametrised by a pair of integers and associated to
  a 3-manifold with torus boundary. In this note, we explain the structure of the
  asymptotic expansions of the 3D-index when $q=e^{2\pi i\tau}$ and $\tau$ tends to
  zero (to all orders and with exponentially small terms included), and discover
  two phenomena: (a) when $\tau$ tends to zero on a ray near the positive real axis,
  the horizontal asymptotics of the meromorphic 3D-index match to all orders with
  the asymptotics of the Turaev--Viro invariant of a knot, in particular explaining
  the Volume Conjecture of Chen--Yang from first principles, (b) when $\tau \to 0$
  on the positive imaginary axis, the vertical asymptotics of the 3D-index
  involves periods of a plane curve (the $A$-polynomial), as opposed to algebraic
  numbers, explaining some predictions of Hodgson--Kricker--Siejakowski and leading
  to conjectural identities between periods of the $A$-polynomial of a knot and
  integrals of the Euler beta-function.
\end{abstract}

\maketitle

{\footnotesize
\tableofcontents
}


\section{Introduction}
\label{sec.intro}

\subsection{The 3D-index}
\label{sub.two}

Quantum invariants attached to 3-dimensional objects, whether defined as
functions at complex roots of unity, $q$-series, or by analytic functions
in the cut plane $\BC'=\BC\setminus(-\infty,0]$, have many interesting and surprising
connections with each other. Although they are well-defined topological invariants,
their relations are largely conjectural and lead to startling statements and numerical
predictions explained in detail in work of Gukov, Mari\~{n}o and
collaborators~\cite{Gukov:BPS, Gukov:resurgence, Gukov:largeN, GGM, GGM:peacock} and
in two papers of Zagier and the first author~\cite{GZ:kashaev,GZ:qseries}.

In this paper, we will focus on the asymptotic properties of a very interesting quantum
knot invariant, the 3D-index of Dimofte--Gaiotto--Gukov~\cite{DGG1,DGG2}
and its meromorphic version of~\cite{GK:mero}. The 3D-index is a collection of
$q$-series with integer coefficients associated to an ideally triangulated 3-manifold
with torus boundary components~\cite{DGG1,DGG2}. These $q$-series occur in mathematical
physics as BPS counts of sypersymmetric field theories and have fascinating
properties. They are topological invariants ~\cite{GHRS,GK:mero} that are
conjecturally related to other quantum invariants such as the complex Chern--Simons
partition function (also known as the state-integrals of Andersen--Kashaev
~\cite{AK}) and even the Kashaev invariant itself~\cite{K95}.

One of the discoveries of~\cite{GZ:qseries} is that the radial asymptotics of
$q$-hypergeometric series in general depend on the ray in which $q$ approaches
$1$ or any complex root of unity. A further discovery of~\cite{GGM, GGM:peacock}
is that as we move the ray, the asymptotic series changes by linear combinations of
power series in $\tq=e^{-2\pi i/\tau}$ (when $q=e^{2\pi i\tau}$) with integer coefficients.



\subsection{Two discoveries for the asymptotics of the 3D-index
  of the $4_1$ knot}

The paper concerns two numerical discoveries concerning the horizontal and the vertical
asymptotics of the meromorphic 3D-index:
\begin{itemize}
\item[(a)] the horizontal asymptotics of the meromorphic
  3D-index $\Imer(0,0)(q)$ matches those of the Turaev--Viro invariant of a knot
  and involves algebraic coefficients,
  \item[(b)] the vertical asymptotics of the meromorphic
  3D-index involve periods (in the sense of Kontsevich--Zagier~\cite{KZ:periods})
  of the $\PSL_2(\BC)$-character variety.
\end{itemize}
Both discoveries stem from the fact that the meromorphic 3D-index is a sum over
the integers of the rotated 3D-index, the latter being a bilinear combination of
colored holomorphic blocks. The first discovery explains the Chen--Yang volume
conjecture~\cite{CY} from first principles and refines it to all orders in
perturbation theory and the second explains some predictions of
Hodgson--Kricker--Siejakowski for the asymptotics of the meromorphic
3D-index~\cite{HKS}.

We will present these discoveries in the order they were found for the $4_1$ knot. 
Our first experiment was to compute the first few terms of the asymptotics of the
meromorphic 3D-index for the $4_1$ knot,
\be
\label{41.mer}
\Imer_{4_1}(0,0)(q) =
1 - 4q - q^2 + 36q^3 + 70q^4 + 100q^5 + 34q^6 - 116q^7 - 410q^8 - 808q^9+\dots\,,
\ee
and found out that
\be
\label{Imernum}
\begin{aligned}
  \Imer_{4_1}(0,0)(q)
  &\sim e^{\frac{2 \Vol(4_1)}{2 \pi \tau}}
\frac{1}{3^{\frac{3}{4}}2^{\frac{1}{2}} \sqrt{\tau}} \\
  &\hspace{-0.5cm}\times\left(1-\frac{19}{24\sqrt{-3}^3}2\pi i\tau
    +\frac{1333}{1152\sqrt{-3}^6}(2\pi i\tau)^2
    -\frac{1601717}{414720\sqrt{-3}^9}(2\pi i\tau)^3+\dots\right) 
\end{aligned}
\ee
when $q=e^{2\pi i\tau}$ and $\tau \to 0$ on a ray sufficiently close to the
positive real axis. We then recognised that the number 1333 appeared in the 
asymptotics of the Turaev--Viro invariant $\TV_{4_1,m+1/2}$ (computed years ago
in unpublished work of the first author), and in fact the above series agreeded
with the asymptotics of the Turaev--Viro invariant $\TV_{4_1,m+1/2}$, after
merely replacing $\tau$ by $m+1/2$.
Surely, this was not a coincidence, and the matching was then checked for the next
7 terms, and still held. The Turaev--Viro invariant is a complex-valued function
at roots of unity defined by a state-sum using an ideal triangulation of the knot
complement, placing representations of $\mathfrak{sl}_2(\BC)$ at the edges, using
the quantum 6j-symbols at each tetrahedron, and summing over the labels. 

The observed matching horizontal asymptotics of the meromorphic 3D-index and of the
Turaev--Viro invariant, aside from explaining a volume conjecture, suggests that
the two invariants are realisations of the same quantum invariant, one being
a $q$-series with integer coefficients, and another being a function at roots of unity.
Moreover, it hints that the Turaev--Viro invariant is one entry of a matrix-valued
quantum knot invariant whose rows and columns are labeled by boundary-parabolic
$\PSL_2(\BC)$-representations of the knot complement. A third realisation of this
invariant is the bilinear combination of asymptotic series, and the three realisations
are related by an underlying holomorphic function on a cut plane
$\BC':=\BC\setminus (-\infty,0]$, much in the theme of~\cite{GZ:kashaev}
and~\cite{GZ:qseries}. 

Our second discovery is about the vertical asymptotics of the meromorphic 3D-index
$\Imer_{4_1}(0,0)(q)$ of the $4_1$ knot. This time as $\tau \in i \BR_+$ tends to
$0$ on the imaginary axis, we found out that
\be
\begin{tiny}
\label{Imernumver}
\begin{aligned}
  \Imer_{4_1}(0,0)(q)
  &\sim
  e^{\frac{2 \Vol(4_1)}{2 \pi \tau}}
\frac{1}{3^{\frac{3}{4}}2^{\frac{1}{2}} \sqrt{\tau}}  
  \left(1-\frac{19}{24\sqrt{-3}^3}2\pi i\tau
    +\frac{1333}{1152\sqrt{-3}^6}(2\pi i\tau)^2
    -\frac{1601717}{414720\sqrt{-3}^9}(2\pi i\tau)^3+\dots\right)\\
  &-i
  e^{\frac{-2 \Vol(4_1)}{2 \pi \tau}}
\frac{1}{3^{\frac{3}{4}}2^{\frac{1}{2}} \sqrt{\tau}}  
  \left(1+\frac{19}{24\sqrt{-3}^3}2\pi i\tau
    +\frac{1333}{1152\sqrt{-3}^6}(2\pi i\tau)^2
    +\frac{1601717}{414720\sqrt{-3}^9}(2\pi i\tau)^3+\dots\right)\\
  &+ \kappa_{4_1} \frac{i}{\tau}+ \kappa_{4_1}' i \tau +\dots\,.
\end{aligned}
\end{tiny}
\ee
The first two series in~\eqref{Imernumver} are exactly those appearing
in~\eqref{Imernum}, but the third series is a new one starting with
$\kappa_{4_1}/(i\tau)$ for an unknown constant 
\be 
\label{k41num}
\kappa_{4_1} = 0.4458257949935614977\dots\,.
\ee
From our previous experience, we tried to
recognise this constant as an algebraic number, after perhaps multiplying it
with a small power of the square root of $\pi$. But all our attempts failed.
With considerable effort, we computed the number $\kappa_{4_1}$ to higher precision,
but the recognition
failed again. We then thought to use a different version of the asymptotic series,
a power series in $\hbar$ with coefficients rational functions of the $A$-polynomial
curve. Doing so, the summation over $n \in \BZ$ corresponds to integration of the
1-loop term of the above series with respect to $u$, as in explained in detail in
Section~\ref{sub.41vertical}. Much to our surprise, we found out that
$\kappa_{4_1}=(\varpi_{4_1}+\overline{\varpi}_{4_1})/(2\pi)$ where
\be
\label{k41exact}
  \varpi_{4_1} =\int_{\BR}\frac{1}{\sqrt{e^{-2u} -2 e^{-u}-1-2 e^u + e^{2u}}}du \,.
\ee
This explained our failure to recognise $(\kappa_{4_1}/2\pi)$ as an algebraic number.
Instead, it is a period of the $A$-polynomial curve, where the word period refers to
the countable subring of the complex numbers introduced by
Kontsevich--Zagier~\cite{KZ:periods}. What is more, the above discovery tells us that
the vertical asymptotics of the rotated 3D-index involve, in addition to the asymptotic
series with algebraic coefficients, further power series with period coefficients!
Continuing our experiment with $4_2$, and using the exact value of $\kappa_{4_1}$
from before, we numerically computed the coefficient
$\kappa_{4_1}'$ of $i \tau$ in~\eqref{Imernumver} and found it to be
\be
\kappa_{4_1}' = 0.10059754907380012789 \dots\,.
\ee
We then theoretically computed the next term in the asymptotic series (which
was an explicit rational function of $e^u$) and 
checked that $\kappa_{4_1}'=2\pi(\varpi_{4_1}'+\overline{\varpi}_{4_1}')$
where
\be
\label{k41exactp}
\varpi_{4_1}' =\int_{C}
\frac{e^{-3u} - e^{-2 u} - 2 e^{-u} + 5 - 2 e^{u} - e^{2 u} + e^{3 u}}{
  (e^{-2u} -2 e^{-u}-1-2 e^u + e^{2u})^{7/2}} du \,,
\ee
for a suitable contour $C$.
Although the above identity would be impossible to guess from its numerical
values, this was hardly a surprise, and confirmed the theory that we will discuss
at a later section. 

This completed our numerical experiments with the $4_1$ knot, which admittedly has
a simple trace field $\BQ(\sqrt{-3})$ and a rather easy second degree $A$-polynomial.
To test our ideas further, we then looked for a knot with more interesting trace
field, i.e., the $5_2$ knot with cubic trace field and a cubic $A$-polynomial.
Repeating the analysis for the meromorphic 3D-index of the $5_2$ knot we found that,
\be
\label{52.mer}
\Imer_{5_2}(0,0)(q) =
   1 - 8q + 21q^2 + 92q^3 + 80q^4 - 95q^5 - 546q^6 - 1092q^7 - 1333q^8 - 756q^9+\dots\,,
\ee
and that when $\tau \in i \BR_+$ tends to zero, we have a similar expansion
to~\eqref{Imernumver} involving two known asymptotic series with algebraic
coefficients, as well as a third series that starts with $i\kappa_{5_2}/\tau + O(1)$
where the new constant is numerically given by 
\be 
\label{kappa52num}
\kappa_{5_2} = 0.3973476532990492346\dots\,.
\ee
We then discovered that this constant is also a period of the $A$-polynomial curve. 
Let us summarise our findings here. Let $y_j(x)$ for $j=1,2,3$ denote the three
roots of the equation~\cite[Eqn.(233)]{GGM:peacock}
\be
\label{p52}
y^2 = (1 - y)(1 - x y)(1 - y/x) \,.
\ee
These are algebraic functions of $x$ that collide when $x$ is a root of
the discriminant of~\eqref{p52}
$x^8 - 6 x^7 + 11 x^6 - 12 x^5 - 11 x^4 - 12 x^3 + 11 x^2 - 6 x + 1$ with respect
to $y$. This discriminant has two real roots at $a_0=0.235344\dots$ and
$a_1=4.249090\dots$ that satisfy $a_0 a_1=1$.
All three branches $y_j(x)$ are real when $x>a_1$, and they are chosen so that $y_3(x)$
real for all $x$ in $\BR$ and $y_1(x)$ and $y_2(x)$ complex-conjugate when $x$
is real with
\be
y_1(a_1+1)= 1.433146\dots, \qquad
y_2(a_1+1)= 3.823982\dots, \qquad
y_3(a_1+1)= 0.182470\dots\,.
\ee
Consider the rational function $1/\delta_{5_2}(x,y)$ on the affine
curve~\eqref{p52} where~\cite[Eqn.(235)]{GGM:peacock}
\be
\label{delta52}
\delta_{5_2}(x,y) = -y + (1 + x + 1/x)/y - 2/y^2 \,,
\ee
and let $f_j(x)=1/\delta_{5_2}(x,y_j(x))$ denote the three branches
for $j=1,2,3$. Since $y_j(x)=y_j(1/x)$, it follows that $f_j(x)=f_j(1/x)$ for all
$x$ and $j=1,2,3$. Consider the periods
\be
\label{k52exact}
\varpi^{(j)}_{5_2} = \int_{\log(a_1)}^\infty f_j(e^u) du, \qquad j=1,2,3, \qquad
\varpi^{(3+j)}_{5_2} = \int_{\log(a_0)}^{\log(a_1)} f_j(e^u) du \,.
\ee
These periods are given by absolutely convergent integrals and
their numerical value can easily be computed to high precision (e.g., to 500 digits),
and given to 20 digits by
\be
\label{k52num}
\begin{tiny}
\begin{aligned}
\varpi^{(1)}_{5_2} &= 0.50571950675093952382\dots\,, &
\varpi^{(4)}_{5_2} &= 0.23686545502355828387\dots - i1.6144780233538382224\dots\,, 
\\
\varpi^{(2)}_{5_2} &= -0.47549190713818022860\dots\,, &
\varpi^{(5)}_{5_2} &= 0.23686545502355828387\dots + i1.6144780233538382224\dots\,,
\\
\varpi^{(3)}_{5_2} &= -0.030227599612759295211\dots\,, &
\varpi^{(6)}_{5_2} &= -0.47373091004711656775\dots\,.
\end{aligned}
\end{tiny}
\ee
%
The above periods satisfy the relations
\be
\label{52prelations}
\varpi^{(1)}_{5_2} + \varpi^{(2)}_{5_2} + \varpi^{(3)}_{5_2} =0,
\qquad
\varpi^{(4)}_{5_2} + \varpi^{(5)}_{5_2} + \varpi^{(6)}_{5_2} =0 \,,
\ee
which is a consequence of the fact that $f_1(x)+f_2(x)+f_3(x)=0$ for all $x$. 

We then found numerically that 
\be
\label{k52guess}
\begin{aligned}
  \kappa_{5_2}
  &=
   \frac{1}{2\pi} (4 \varpi^{(1)}_{5_2}
   + \varpi^{(4)}_{5_2}+\varpi^{(5)}_{5_2})\,.
\end{aligned}
\ee
Our constant is different from the number
$\kappa^{\mathrm{HKS}}_{5_2}=.534186\dots$ of~\cite[Fig.10.2]{HKS}, given to three digits
in the above reference, but three more digits were also given to us by A. Kricker.
We found out that the latter constant also agrees numerically with a period of the
$A$-polynomial (for the real branch),
\be
\label{hks52guess}
\kappa^{\mathrm{HKS}}_{5_2} = -2 \varpi^{(3)}_{5_2} - \varpi^{(6)}_{5_2}
= -\int_\BR f_1(e^u) du \,.
\ee
This leads to a conjectured identity
\be
\label{52periodconj}
\frac{1}{(2\pi i)^2} \int_{i \BR^2}
\mathrm{B}^2\bigl(\tfrac{1}{2}-x,\tfrac{1}{3}+x-y\bigr)
\mathrm{B}\bigl(\tfrac{1}{2}-x,\tfrac{1}{3}+2y\bigr) dx dy
=
-\int_\BR f_1(e^u) du\,,
\ee
(which we have checked to around 20 digits), where
$\mathrm{B}(x,y)=\Gamma(x)\Gamma(y)/\Gamma(x+y)$ is Euler's beta function.
The above identity implies that the double integral on left hand-side is a period
of an elliptic curve~\eqref{p52} over $\BQ$. 


Identities of the above form should be valid for all hyperbolic knots.

\subsection{Horizontal asymptotics and the Turaev--Viro invariant}
\label{sub.conj}

Our observations discussed in the previous section
lead conjectures for the horizontal and the vertical
asymptotics of the meromorphic 3D-index, which not only allow the effective
computation of the asymptotic series, but also explain their structure.  

The ingredients of first conjecture are: the meromorphic 3D-index
(defined in~\cite{GK:mero} and briefly reviewed in Section~\ref{sub.mero} below),
the Turaev--Viro invariant of a knot (defined in~\cite{CY,DKY} and recalled in
Section~\ref{sec.TV}), the asymptotic series $\wh\Phi^{(\s)}_n(\hbar)$
for $s=\s_1$ and $s=\s_2$ of the form
\be
\label{Phisnh}
\wh\Phi^{(\s)}_n(\hbar) = e^{\frac{V_{\s}}{\hbar}} \Phi^{(\s)}_n(\hbar),
\qquad \Phi^{(\s)}_n(\hbar) \in
\frac{1}{\sqrt{\delta_{\s}}}(1+ F_{\s}[[n,\hbar]])
\ee
defined in Section~\ref{sec.phi} below and agreeing when $\s=\s_1$ and $n=0$
with the asymptotic
series of the Kashaev invariant of a knot. The labels $\s_1$ and $\s_2$ refer to
the geometric $\PSL_2(\BC)$-representation $\s_1$ of the hyperbolic knot and
its complex-conjugate representation. The final ingredient is a summation over
the integers. 

\begin{conjecture}
  \label{conj.1}
For every hyperbolic knot $K$ in $S^3$, there exists a computable series
$\wh\Phi^{(\s_1)}_n$ which gives the asymptotic expansion of the meromorphic 3D-index
(as $\tau \to 0$ on a ray sufficiently close to the positive real numbers),
of the Turaev--Viro invariant (as $m \to \infty$)
  \begin{align}
  \label{main1a}
  \Imer_{K}(0,0)(e^{2 \pi i \tau}) & \sim
  \sum_{n \in \BZ} \wh\Phi^{(\s_1)}_{K,n}(2 \pi i \tau)
  \wh\Phi^{(\s_2)}_{K,n}(-2 \pi i \tau) \,, \\
  \label{main1b}
    \TV_{K,m+1/2} & \sim \sum_{n \in \BZ} 
    \wh\Phi^{(\s_1)}_{K,n}(2 \pi i/(m+1/2))
    \wh\Phi^{(\s_2)}_{K,n}(-2 \pi i/(m+1/2)) \,, 
\end{align}
and of the rotated 3D-index (as $\tau \to 0$ on a ray sufficiently close to
$\BR_+$):
\be
\label{main1c}
\Irot(n,n')(e^{2 \pi i \tau}) \sim
\wh\Phi^{(\s_1)}_n  (2\pi i \tau) \wh\Phi^{(\s_2)}_{n'}(-2\pi i \tau) \,.
\ee
\end{conjecture}
This conjectured was verified numerically for the $4_1$ and the $5_2$ knots. 

We make two remarks. The first is that the meromorphic 3D-index is a sum over
the integers of the rotated 3D-index (Equation~\eqref{ImerIrot}),
the latter being a bilinear combination of colored holomorphic blocks
(Equation~\eqref{IrotH2}) whose asymptotics are expressed linearly in terms of a
matrix of $\wh\Phi(\hbar)$-series (Equation~\eqref{HPhi}). Thus, the asymptotics of the
meromorphic 3D-index are bilinear combinations of asymptotic series. But when $\tau$
tends to zero nearly horizontally, one term $(\s_1,\s_2)$ of this bilinear combination
exponentially dominates all others, thus explaining Equation~\eqref{main1a}.
On the other hand, the Turaev--Viro invariant of a knot is Reshetikhin--Turaev
invariant of the double of the knot complement, thus doubling occurs already on the
level of topology.

The second is that the above conjecture can be explained by a conjectural identity
among analytic functions, one being the Borel resummation of the asymptotic
series $\wh\Phi^{(\s)}_n(2\pi i\tau)$ and a second being a state-integral;
see Section~\ref{sub.asyholo} below. When $\tau$ is nearly horizontal, 
one can ignore all $\tq$-corrections and keep only the dominant bilinear
combination of asymptotic series, leading to the asymptotic statements
~\eqref{main1a} and~\eqref{main1c} of Conjecture~\ref{conj.1}.

\subsection{Vertical asymptotics and periods}
\label{sub.vertical}

In this section, we discuss the vertical asymptotics (denoted by
$\tau \downarrow 0$) of the 3D-index, where $\tau$ tends to 0 on the positive
imaginary axis. In this case, we ignore the $\tq$-corrections which are theoretically
unknown and computationally inaccessible. 
A new ingredient is a collection of asymptotic series parametrised by a labeling
set of boundary parabolic $\PSL_2(\BC)$-representation.\footnote{There are subtleties
involved since the latter space can be positive dimensional, however we
(and~\cite{GZ:kashaev,HKS}) will ignore such subtleties.} These series appear
as follows. The rotated 3D-index is a bilinear combination of colored holomorphic
blocks. The vertical asymptotics of the latter are linear combinations of asymptotic
series $\wh\Phi^{(\s)}_n  (2\pi i\tau)$ where $\s$ denotes a boundary parabolic
$\PSL_2(\BC)$-representation. It follows that the vertical asymptotics of the
rotated 3D-index are \emph{bilinear} combinations of asymptotic series. However,
among bilinear combinations of all pairs $(\s,\s')$, only the pairs $(\s,\bar\s)$
contribute, where $\bar\s$ denotes the complex-conjugate of the representation
corresponding to $\s$. 

\begin{conjecture}
  \label{conj.2}
  When $\tau \downarrow 0$ vertically, we have
\be
\label{main2}
\begin{aligned}
\Irot(n,n')(e^{2 \pi i \tau}) \sim & \sum_{\s} \ve_{\s}
\wh\Phi^{(\s)}_n  (2\pi i \tau) \wh\Phi^{(\bar\s)}_{n'}(-2\pi i \tau) \,,
\end{aligned}
\ee
with $\ve_\s=-1$ when $\delta_\s <0$ and $1$ otherwise.
\end{conjecture}

Note that after removing the exponential factors, the constant terms of the power
series in $\tau$ in~\eqref{main2} will all be positive numbers. Moreover,
the asymptotic series are polynomially bounded with respect to $\tau$--perhaps
a consequence of a unitarity property of the rotated 3D-index.

We next discuss the vertical asymptotics of the meromorphic 3D-index, ignoring
as above any $\tq$-corrections. 
Here, we find a surprise and a new phenomenon: the coefficients of the asymptotic
series are no longer algebraic numbers, but periods (in the sense of
Kontsevich--Zagier~\cite{KZ:periods}) on a plane curve which is none other than
the $A$-polynomial of the knot~\cite{CCGLS}. 
To phrase our conjecture, instead of the power series $\Phi^{(\s)}_n(\hbar)$
in $\hbar$ with coefficients polynomials in $n$, we will use the power series
$\Phi(x,y,\hbar)$ in $\hbar$ with coefficients rational functions in the
$A$-polynomial curve. The latter are the $x$-deformed asymptotic series
studied in~\cite{GGM:peacock}. The relation between the two types of series is
given by $\Phi^{(\s)}_n(\hbar) = \Phi^{(\s)}(e^{n\hbar},\hbar)$ where
$\Phi^{(\s)}(x,\hbar)=\Phi(x,y_\s(x),\hbar)$ and $y_\s=y_s(x)$ is a branch of the
$A$-polynomial curve.

\begin{conjecture}
\label{conj.3}
When $\tau \downarrow 0$ vertically, we have
\be
\label{main3}
\begin{aligned}
  \Imer(0,0)(e^{2 \pi i \tau}) \sim & \sum_{\s \,\, \mathrm{complex}}
  \sum_{n \in \BZ}
  \wh\Phi^{(\s)}_n  (2\pi i \tau) \wh\Phi^{(\bar\s)}_n(-2\pi i \tau)
  \\ & + \frac{1}{2\pi i \tau} \sum_{\s}
  \int_{C_\s} \wh\Phi^{(\s)} (e^{\tau u}, 2\pi i \tau)
  \wh\Phi^{(\s)} (e^{\tau u},-2\pi i \tau) du\,.
\end{aligned}
\ee
\end{conjecture}
The above conjecture explains the shape of the vertical asymptotics of the
meromorphic 3D-index found by~\cite{HKS}. A bit of mystery remains: our vertical
asymptotics include more terms that those of~\cite{HKS}. For example, for the
$4_1$ knot, the series starting with the term $\kappa_{4_1}$ in
Equation~\eqref{Imernumver} is missing from~\cite[Sec.10.1.1]{HKS}, and for the
$5_2$ knot, our constant $\kappa_{5_2}$ differs from the one
of~\cite[Fig.10.2]{HKS}. We do not understand this discrepancy which involves
some global differences in the contours appearing in Equation~\eqref{main3}.

The paper is organised as follows.

In Section~\ref{sec.3d}, we introduce a rotated
form of the 3D-index (see Equation~\eqref{eq.Irot} below) which conveniently 
decouples two commuting actions of the $q$-Weyl algebra, and also expresses the
meromorphic 3D-index as a trace of the rotated 3D-index
(see Equation~\eqref{ImerIrot} below). We next express the rotated 3D-index
bilinearly in terms of colored holomorphic blocks $h^{(\a)}_n(q)$. These blocks
(which are $q$-hypergeometric series) can be defined and computed either from
a factorisation of a state-integral, or from a linear $q$-difference equation and
their initial values at $n=0$ given in~\cite{GZ:qseries}. A side bonus of the colored
holomorphic blocks is an effective computation of the $q$-series expansion of
the rotated and of the meromorphic 3D-index.

We then compute the asymptotics of the holomorphic blocks in terms of (completed)
power series $\Phi^{(\s)}_n(h)$. These power series can be defined and computed using
either formal Gaussian integration~\cite{DG} (which amounts to a stationary phase
expansion of a state-integral), or from a linear $q$-difference equation and their
initial values at $n=0$ given in~\cite{GZ:kashaev}. In particular, the dominant series
$\Phi^{(\s_1)}(h)$ is exactly the asymptotic expansion of the Kashaev invariant to
all orders in $h=2\pi i/N$~\cite{GZ:kashaev}. 

Finally, we compute the radial asymptotics of the rotated 3D-index in terms of bilinear
combinations of the completed power series $\Phi^{(\s)}_n(h)$, which, after a formal
Gaussian integration with respect to $n$, give the horizontal asymptotics of the
rotated 3D-index.

Analysing further the possible bilinear asymptotics that can appear, combined with
the fact that the meromorphic 3D-index is bounded when $\tau$ tends to zero vertically,
allows us to compute the horizontal asymptotics of the meromorphic 3D-index.

We end this introduction with a remark on rigour. Although the 3D-index
(in its original, or rotated form) is a well-defined collection of
$q$-series~\cite{DGG1,DGG2} which was shown to be a topological invariant of hyperbolic
knots in two different ways~\cite{GHRS, GK:mero}, its factorisation in terms of
holomorphic blocks or colored holomorphic blocks discussed in Sections~\ref{sub.holo}
and~\ref{sub.colored} are conjectural, and so are our asymptotic statements which
form the main body of this paper. Likewise, the Turaev--Viro invariant of a knot
is a well-defined topological invariant, however the conjecture on its
all-orders asymptotic expansion is new, and unproven. However, the stated conjectures
for the 3D-index and for the Turaev--Viro invariant are numerically testable
(and have been tested for the case of the $4_1$ and the $5_2$ knots), and
lead to surprising predictions as well as an explanation of the structure and
computation of the asymptotics of these quantum invariants.


\section{A review of the 3D-index}
\label{sec.3d}

In this section, we review the 3D-index in various forms: the original one
of~\cite{DGG1, DGG2}, the rotated one, the one coming from colored holomorphic
blocks, and the meromophic version~\cite{GK:mero}. 

In their seminal papers, Dimofte--Gaiotto--Gukov introduced the 
3D-index~\cite{DGG1, DGG2}, which comes from the low energy limit of an
$N=2$ sypersymmetric conformal theory in 6 dimensions, compactified in 3-dimensions. 
More concretely, the 3D-index is a collection of Laurent 
$q$-series with integer coefficients attached to every integer homology 
class of the boundary of an ideally triangulated 3-manifold with 
torus boundary components. It is known that this collection is a 
topological invariant of hyperbolic 3-manifolds~\cite{GHRS}, and that 
it can be assembled to a meromorphic function on a complex 
torus~\cite{GK:mero} which, too, is a topological invariant. 

In a further direction, Beem--Dimofte--Pasquetti studied 
the relationship between the 3D-index and a vector space of holomorphic 
blocks~\cite{holo-blocks}.

We next discuss several equivalent formulations of the 3D-index, namely the
fugacity version, the rotated version, the meromorphic version, and the
relation with holomorphic blocks. 


\subsection{The rotated 3D-index}
\label{sub.rot}

Recall that the 3D-index of a 3-manifold with a torus boundary component (marked
by a meridian and a longitude) is a collection of $q$-series parametrised by a pair
of integers $(m,e)$. It turns out that the 3D-index is annihilated by two actions of
the $q$-Weyl algebra, which are acting in the variables $(m,e)$ in a coupled way.
The rotated 3D-index is also an equivalent collection of $q$-series where the two
actions are decoupled. To define the rotated 3D-index, we suppress the ambient
3-manifold $M$ writing $I(m,e)(q)$ instead of $I_M(m,e)(q) \in \BZ((q^{\frac{1}{2}}))$. 
Recall two commuting actions of the $q$-Weyl algebra~\cite[Eqn.2.29a]{DGG1}:


\be
\label{eq.LM}
\begin{aligned}
L_+ I(m,e)(q) &= q^{\frac{e}{2}} I(m+1,e)(q)\,, & \qquad M_+ I(m,e)(q) &= 
q^{\frac{m}{2}} I(m,e-1)(q)\,, \\  
L_- I(m,e)(q) &= q^{\frac{e}{2}} I(m-1,e)(q)\,, & \qquad M_- I(m,e)(q) &= 
q^{\frac{m}{2}} I(m,e+1)(q) \,.
\end{aligned}
\ee
It is easy to see that the four operators $L_{\pm}$ and $M_{\pm}$
commute except in the following instance:
$$
L_+ M_+ = qM_+ L_+\,, \qquad L_- M_- = q^{-1} M_- L_- \,.
$$
The two commuting actions of the $q$-Weyl algebra mix the coordinates of
the 3D-index. A separation of the two actions is possible and this
motivates the {\em rotated form} of the 3D-index given by
\be
\label{eq.Irot}
\Irot(n,n')(q) = \sum_{e \in \BZ} I(n-n',e)(q) q^{\frac{e(n+n')}{2}} \,.
\ee
When the ideal triangulation has a strict angle structure, it can be shown
that the above series is well-defined. In the above equation, and the ones that
follow, convergence is ensured if the minimum $q$-degree of $I(m,e)(q)$ is a
positive quadratic function of $(m,e)$ bounded below by $O(m^2+e^2)$ except at
finitely many directions of linear growth $O(|m|,|e|)$ (parallel to the normal
vectors of the Newton polygon of the $A$-polynomial). This convergence is deduced
from the results of~\cite{GHRS,GHHR}.


\begin{lemma}
\label{lem.rot}
We have:
\be
\label{eq.LM2}
\begin{aligned}
L_+ \Irot(n,n')(q) &= \Irot(n+1,n')(q)\,, & \qquad M_+ \Irot(n,n')(q) &= 
q^{n} \Irot(n,n')(q)\,, \\  
L_- \Irot(n,n')(q) &= \Irot(n,n'+1)(q)\,, & \qquad M_- \Irot(n,n')(q) &= 
q^{-n'} \Irot(n,n')(q) \,.
\end{aligned}
\ee
\end{lemma}

\begin{proof}
We have:
\begin{align*}
L_+ \Irot(n,n')(q) &= \sum_e q^{\frac{e}{2}} I(n+1-n',e)(q) q^{\frac{e(n+n')}{2}} \\
&= \sum_e I(n+1-n',e)(q) q^{\frac{e(n+1+n')}{2}} = \Irot(n+1,n')(q)\,,
\end{align*}
and
\begin{align*}
M_+ \Irot(n,n')(q) &= \sum_e q^{\frac{n-n'}{2}} I(n-n',e-1)(q) q^{\frac{e(n+n')}{2}} 
= \sum_e I(n-n',e-1)(q) q^{\frac{e(n+n')+n-n'}{2}} \\
&= \sum_e I(n-n',e)(q) q^{\frac{(e+1)(n+n')+n-n'}{2}} 
= \sum_e I(n-n',e)(q) q^{\frac{e(n+n')}{2}} q^{n} 
\\ &= q^n \Irot(n,n')(q) \,.
\end{align*}
The case of $(L_-,M_-)$ is similar.
\end{proof}

Note that $\Irot_K(0,0)(q)$ coincides with the invariant denoted by
$I^{\mathrm{tot}}_K(q)$ in Equation (2) of~\cite{GHRS}, and with the invariant
studied in~\cite{GZ:qseries}.

We now discuss some symmetries of the 3D-index and its rotated version. The symmetry,
\be
\label{Isym0}
I(m,e)(q)=I(-m,-e)(q)\,,
\ee
of the 3D-index implies one for the rotated version $\Irot(n,n')(q)=\Irot(-n,-n')(q)$.
The rotated 3D-index has two more symmetries
\be
\label{Isym}
\Irot(n,n')(q)=\Irot(n,-n')(q)=\Irot(-n,n')(q)
=\Irot(-n,-n')(q) \,,
\ee
and 
\be
\label{Ikmsym}
\Irot(n,n')(q^{-1}) = \Irot(n',n)(q) \,,
\ee
which are manifest from the expression of the rotated 3D-index in terms of
colored holomorphic blocks discussed below.

We next discuss how to express the original 3D-index from its rotated version,
reversing Equation~\eqref{eq.Irot}.

\begin{lemma}
\label{lem.rot2}  
For all integers $m$ and $e$ we have:
\be
\label{rot2}
I(m,e)(q) = q^{-\frac{me}{2}} \sum_{k \in \BZ} \Irot(m+k,k)(q) q^{-ke} \,.
\ee
\end{lemma}

\begin{proof}
We have:
\begin{align*}
  \sum_k \Irot(m+k,k)(q) q^{-k e} &= \sum_{k,e'} I(m,e')(q) q^{\frac{e'm}{2}+k(e'-e)} =
  \sum_{e'} q^{\frac{e'm}{2}} I(m,e')(q) \sum_k q^{k(e'-e)} \\ &
  = \delta_{e,e'} \sum_{e'} q^{\frac{e'm}{2}} I(m,e')(q) = q^{\frac{e m}{2}} I(m,e)(q) \,.
\end{align*}
For completeness, we can also compute the action of the operators $L_\pm$ and
$M_\pm$ on the right hand side $J(m,e)(q)$ of Equation~\eqref{rot2}. We have:
\begin{align*}
  L_+ J(m,e)(q) &= q^{-\frac{me}{2}} \sum_{k \in \BZ} L_+ \Irot(m+k,k)(q) q^{-ke} =
  q^{-\frac{me}{2}} \sum_{k \in \BZ} \Irot(m+k+1,k)(q) q^{-ke} 
\\ &= q^{\frac{e}{2}} J(m+1,e)(q) \,,
\end{align*}
and
\begin{align*}
  M_+ J(m,e)(q) &= q^{-\frac{me}{2}} \sum_{k \in \BZ} M_+ \Irot(m+k,k)(q) q^{-ke} =
  q^{-\frac{me}{2}} \sum_{k \in \BZ} q^{m+k} \Irot(m+k,k)(q) q^{-ke} 
\\ &= q^{\frac{m}{2}} J(m,e-1)(q) \,.
\end{align*}
The case of $L_-$ and $M_-$ is similar.
\end{proof}


\subsection{The meromorphic 3D-index}
\label{sub.mero}

We now discuss another realisation of the 3D-index as a meromorphic function.
In~\cite{GK:mero}, it was shown that if a triangulation has a strict
angle structure, then the following function
\be
\label{ImerI}
\Imer(z,w)(q) = \sum_{m,e \in \BZ} I(m,e)(q) e^{2 \pi i \tau (m z + e w)}
\ee
is meromorphic and regular at $(z,w)=(0,0)$. The relation between the meromorphic
and the rotated 3D-index is given by the next lemma.

\begin{lemma}
  \label{lem.rotmero}
  For all integers $\ell$ and $\ell'$, we have:
  \be
  \label{ImerIrotall}
  \Imer(\ell,\ell')(q) = q^{-2\ell \ell'}
  \sum_{n \in \BZ} q^{2 \ell n} \Irot(n, n-2\ell')(q) \,.
  \ee
\end{lemma}

\begin{proof}
Using Equations~\eqref{ImerI} and~\eqref{rot2},  we have:
  \begin{align*}
    \Imer(\ell,\ell')(q) &= \sum_{e,m} I(m,e)(q) q^{m \ell + e \ell'}
    = \sum_{e,m} q^{-\frac{me}{2}} \sum_k
    \Irot(m+k,k)(q) q^{m \ell + e \ell'-ke} \\
    &=
    \sum_{m,k} q^{m\ell} \Irot(m+k,k)(q) \sum_e q^{-(\frac{m}{2}+k-\ell')e} =
    \sum_{m,k} q^{m\ell} \Irot(m+k,k)(q) \delta_{\frac{m}{2}+k-\ell',0} \,,
  \end{align*}
  where the last identity follows from the fact that $\sum_{e}q^{ae}=\delta_a$
  (interpreted distributionally). 
  When $m$ is odd the delta function vanishes, and for $m=2m'$, we obtain
  that
\begin{align*}
  \Imer(\ell,\ell')(q) &=
  \sum_{m',k} q^{2m'\ell} \Irot(2m'+k,k)(q) \delta_{m'+k-\ell',0} =
  \sum_{m'} q^{2m'\ell} \Irot(m'+\ell',\ell'-m')(q) \\
  &= \sum_{m'} q^{2m'\ell} \Irot(m'+\ell',m'-\ell')(q) =
  \sum_n q^{2n\ell} \Irot(n+\ell',n-\ell')(q)\,, 
\end{align*}
where the next to last equality follows from Equation~\eqref{Isym}.
The result follows.
\end{proof}


In~\cite{HKS}, the asymptotics of $\Imer(0,0)(q)$ when $q \in (0,1)$ tends to $1$
were studied. The relation of this $q$-series and the rotated 3D-index follows
from the lemma above, and is given by
\be
\label{ImerIrot}
  \Imer(0,0)(q) = \sum_{n \in \BZ} \Irot(n, n)(q) \,.
\ee

\subsection{Holomorphic blocks}
\label{sub.holo}

Consider the 3D-index in the so-called fugacity basis~\cite[Eqn.2.5]{DGG1}:
\be
\label{eq.fuga}
\Ifug(\z,m)(q) = \sum_e I(m,e)(q) \z^e \,.
\ee
Its expression in terms of holomorphic blocks is given below Eqn.6.46
of~\cite{holo-blocks}.
Comparing Equations~\eqref{eq.Irot} and~\eqref{eq.fuga}, and keeping in mind
that $\Ifug(\z,m)$ is a meromorphic function of $\z$ with potential singularities
at $\z \in q^\BZ$, it follows that
\be
\label{eq.rotfuga}
\Irot(n,n')(q) = \lim_{x\to 1}\Ifug(q^{\frac{n+n'}{2}} x,n-n')(q) \,.
\ee
Now, according to~\cite[Eqn.1.3]{holo-blocks} (see
also~\cite[Eqn.6.46]{holo-blocks} for a similar statement for state-integrals),
we have
\be
\label{eq.holoblocks}
\Ifug(\z,m)(q) = \sum_{\a} B^{(\a)}(x;q) B^{(\a)}(\tilde x; \tilde q),
\qquad x=q^{\frac{m}{2}}\z, \quad \tilde x = q^{\frac{m}{2}}\z^{-1}, \quad
\tilde q=q^{-1} \,,
\ee
where $B^{(\a)}(x,q)$ are holomorphic blocks, meromorphic functions of $x$
which are defined both inside and outside the unit $q$-disk, and are
annihilated by the homogeneous $\Ahat$-polynomial. 
Equations~\eqref{eq.rotfuga} and~\eqref{eq.holoblocks} imply that
\be
\label{IrotB}
\Irot(n,n')(q) =
\lim_{x\rightarrow1}\sum_\a B^{(\a)}(q^{-n'}x^{-1};q^{-1})B^{(\a)}(q^{n}x;q)\,.
\ee
This is our starting point for the asymptotics of the 3D-index, which in particular
predicts that the rotated 3D-index satisfies the equation
\be
\label{AJ}
\Ahat(M_+,L_+) \Irot = \Ahat(M_-,L_-) \Irot =0 \,. 
\ee

\subsection{Colored holomorphic blocks}
\label{sub.colored}

The holomorphic blocks $B^{(\a)}(x;q)$ are meromorphic functions of $x$ with
poles at $x \in q^\BZ$, which makes the limit in equation~\eqref{IrotB} difficult
to compute. Instead, we will introduce colored holomorphic blocks $h^{(\a)}(q)$
and express the rotated 3D-index as follows:
\be
\label{IrotH}
\Irot(n,n')(q) = \sum_\a B_{\a,\bar\a} h^{(\a)}_{n}(q) h^{(\bar\a)}_{n'}(q^{-1}) \,,
\ee
where $B=B^t \in \GL_r(\BQ)$ and $\a \mapsto \bar\a$ is an involution corresponding
to the complex conjugation of the set of $\s$. 
The colored holomorphic blocks, defined for $|q| \neq 1$ form a fundamental solution
to the linear $q$-difference equation 
\be
\label{AJ2}
\Ahat(M_+,L_+) H(q) = \Ahat(M_-,L_-) H(q^{-1}) =0 \,,
\ee
which can be defined and computed by applying the Frobenius method. In particular,
this implies that the labeling set of $\a$'s consists of the pairs of edges of the
lower Newton polygon of the $A$-polynomial, together with with a root of the edge
polynomial (the latter known to be product of cyclotomic polynomials).

Going one step further, we can define a $\BZ\times\BZ$ matrix $\Irot(q)$ whose
$(n,n')$ entry
is $\Irot(n,n')(q)$, and a $\BZ\times r$ matrix $H(q)$ whose $(n,\a)$ entry is
$h^{(\a)}_n(q)$, and then write the above equation in the matrix form
\be
\label{IrotH2}
\Irot(q) = H(q) B H(q^{-1})^t \,.
\ee
Equations~\eqref{IrotH2} and~\eqref{ImerIrot} give that
\be
\label{ImerH}
\Imer(0,0)(q) = \mathrm{tr}(H(q) B H(q^{-1})^t) \,.
\ee
The symmetry,
\be
\label{hansym}
h^{(\a)}_n(q) = h^{(\a)}_{-n}(q) \,,
\ee
of the colored holomorphic blocks (for all $\a$ and all integers $n$)
and Equation~\eqref{IrotH2} implies the symmetries~\eqref{Isym} and~\eqref{Ikmsym}
for the rotated 3D-index.

To use equation~\eqref{IrotH2} effectively, we need a method to compute 
the colored holomorphic blocks. We can do so either from the polar decomposition of
the holomorphic blocks when the latter are available (this is the case of the $4_1$
and $5_2$ knots), or by applying the Frobenius method (see,~\cite{Ga:quasi})
to the linear $q$-difference equation that they satisfy when the latter is available
(this is also the case of the $4_1$ and $5_2$ knots). Either way, there
is an ambiguity in the normalisation of the colored holomorphic blocks. 


\section{Asymptotics}
\label{sec.phi}

The previous section expressed the rotated and the meromorphic 3D-index of a knot
bilinearly in terms of a matrix of colored holomorphic blocks--see
equations~\eqref{IrotH2} and~\eqref{ImerH}. 
In this section, we discuss the structure of the horizontal asymptotics of
the colored holomorphic blocks, and consequently of the rotated and of the
meromorphic 3D-index. 

\subsection{A matrix of asymptotic series}
\label{sub.asyholo}

We begin by defining a matrix of asympototic series which ultimately expresses
the asymptotics of the colored holomorphic blocks and of the 3D-index, in all of
its forms. 
Let $H(q)$ and $\wh\Phi(h)$ denote the $\BZ\times r$ matrix with
$H(q)_{n,\a}=h^{(a)}_n(q)$ and $\wh\Phi(h)_{n,\a}=\wh\Phi^{(\a)}_n(h)$.
For an integer $n$, let $H_n(q)$ and $\wh\Phi_n(h)$ denote the $r\times r$
matrices whose rows are the $n$th, $n+1$, \dots $n+r-1$ rows of $H(q)$ and
$\wh\Phi(h)$, respectively. The entries of the matrix $\wh\Phi(h)$ are completed
formal power series of the form given in Equation~\eqref{Phisnh} where $\s$
labels a boundary parabolic $\PSL_2(\BC)$-representation of the knot complement,
and $V_\s=i \mathrm{Vol}(\s)+\mathrm{CS}(\s)$ denotes its complexified
volume whose imaginary part is the volume and the real part is the Chern--Simons
invariant~\cite{Neumann:CS}. The Borel resummation of the completed formal power
series $\wh\Phi$, whose definition we omit, will be denoted by $(s\wh\Phi)(h)$
(for a detailed definition, see~\cite{Sauzin:divergent}). The Borel resummation
depends on a ray in a sector that does not contain Stokes lines. In our case, the
sector is in the upper or lower half-plane and close to the positive real axis,
in which case it is conjectured that there are no Stokes lines in such a
sector~\cite{GGM}, and that Borel resummation is well-defined. 

The relation between the colored holomorphic blocks and the asymptotic series is the
following conjectural identity:
\be
\label{HPhi}
H_n(q) = (s \wh\Phi_n)(2\pi i \tau) M(\tq) H_0(\tq) \diag(\tau^{\kappa_\a}),
\qquad (0 < |\arg(\tau)| < \ve)
\ee
valid for some $\ve>0$ (which depends on the knot), where $M(q) \in \GL_r(\BQ(q))$.
The above equality of exact functions allows us to compute the matrix $\wh\Phi_n(h)$
of asymptotic series in several ways:
\begin{itemize}
\item[(a)]
  via numerical asymptotics of the colored holomorphic blocks as $\tau$ tends to
zero nearly horizontally.
\item[(b)]
  via WKB, using that fact that both matrices $H_n(q)$ and $\wh\Phi_n(h)$ are
fundamental solutions of the same linear $q$-difference equation, namely the
$\Ahat$-equation. In fact, it appears that $\wh\Phi_n(h)$ is the unique discrete
WKB solution to the $\Ahat$ equation with initial condition $\wh\Phi_0(h)$.
\item[(c)]
  via the formulas of~\cite{DG}, where a power series was defined by formal Gaussian
integration using the Neumann--Zagier matrices of an ideal triangulation of a knot
complement and a solution to the Neumann--Zagier equations, whose developing map
is $\a$. In that paper, if we let the solution of the Neumann--Zagier equation vary
infinitesimally around $\a$, we obtain a power series in $h$ whose coefficients
are polynomials in $n$.
\item[(d)]
  via the asymptotics of state-integrals that equal to a suitable matrix-valued cocycle.
\end{itemize}

In the present paper, we will mostly use the WKB ansatz since the
linear $q$-difference equation (i.e., the $\Ahat$-polynomial)
of the $4_1$ and the $5_2$ knots are known. Recall that the WKB ansatz
\be
\label{Ahat.ansatz}
\begin{aligned}
  \Phi_{n}(\hbar)
  &=
  \exp\left(\sum_{\ell=-1}^{\infty}S_{\ell}(n\hbar)\right)=\Phi_{0}(\hbar)
  \exp\left(\sum_{\ell=-1}^{\infty}
    \sum_{k=1}^{\infty}a_{k,\ell}(n\hbar)^k\hbar^{\ell}\right) \,.
\end{aligned}
\ee
substitutes the above power series in the linear $q$-difference equation, and
solves for the coefficients, one degree of $\hbar$ at a time.
The Weyl symmetry $n\leftrightarrow -n$ of the linear $q$-difference equation
implies that we get polynomials in $n^2$, thus $a_{k,\ell}=0$ for all odd $k$. 
The first unknown to be determined is $a_{2,-1}$, which satisfies a polynomial
equation obtained by setting $q=1$ in the linear $q$-difference equation,
$q^n=1+u+O(u^2)$ and replacing the shift operator by $1+a_{2,-1} u + O(u^2)$.
This polynomial equation is identical to the polynomial equation for the
cusp-shape $c$ of the hyperbolic knot, obtained by replacing $M=1+u$, $L=1 + cu$
in the $A$-polynomial of a knot and equating the lower power of $u$ arising to zero.
The coincidence of the polynomial equations for $a_{2,-1}$ and of the cusp shape
is consistent with the AJ-Conjecture which states that the specialisation of $q=1$
to the $\Ahat$-polynomial is the $A$-polynomial. 
Once $a_{2,-1}$ is known, the remaining coefficients $a_{k,\ell}$ for $k \geq 2$
are obtained from the WKB ansatz as a solution to a block triangular linear system.

\subsection{Horizontal asymptotics of the rotated 3D-index}
\label{sub.asyrot}

Equation~\eqref{HPhi}, together with Equation~\eqref{IrotH2} implies that
\be
\label{IrotH3}
\Irot(q) = (s \wh\Phi)(2\pi i \tau) M(\tq) \Irot_0(\tq) M(\tq^{-1}) 
(s \wh\Phi)^*(-2\pi i \tau) \,,
\ee
where $*$ means the conjugate transpose. 
The above is an exact identity that gives the $q \to 1$ asymptotics all the $q$-series
$\Irot(n,n')(q)$ to all orders in $\tau$ and with all exponential terms $\tq$ included.
Ignoring the $\tq$ terms, and assuming that $\tau \to 0$ in a fixed ray with
$0 < |\arg(\tau)| < \ve$, we obtain that
\be
\label{IrotH4}
\Irot(q) \sim \wh\Phi (2\pi i \tau) \diag(M_0) (\wh\Phi)^*(-2\pi i \tau) \,,
\ee
for a matrix $M_0 \in \GL_r(\BQ)$. Even further, when $\tau \to 0$ in a ray
$\arg(\tau)=\th_0$ with $\th_0$ positive and sufficiently small, this implies
the asymptotic expansion of Equation~\eqref{main1c}.

\subsection{Horizontal asymptotics and Gaussian integration}
\label{sub.asymer}

Equations~\eqref{ImerH} and~\eqref{main1c} imply that when $\tau \to 0$ in a
ray $\arg(\tau)=\th_0$ with $\th_0$ positive and sufficiently small,
Equation~\eqref{main1a} holds. In this section, we explain how to expand the sums over
the integers of the bilinear asymptotic series in~\eqref{main1a}, as power series
in $\hbar$. 

We will use the form~\eqref{Ahat.ansatz} for the two asymptotic series 
$\Phi_{n}^{(\s)}(\hbar)$ and $\Phi_{n}^{(\s')}(\hbar)$, and abbreviate 
$\b_{k,\ell}=a_{k,\ell}^{(\s)}+(-1)^{k+\ell}a_{k,\ell}^{(\s')}$ and 
$\b=2\b_{2,-1}$. Of course, $\b$ and $\b_{k,\ell}$ depend on $(\s,\s')$. Then,
we have
\be
\label{Phi2}
\begin{aligned}
\Phi_{n}^{(\s)}(\hbar)\Phi_{n}^{(\s')}(-\hbar)
&=
  \exp\left(\sum_{\ell=-1}^{\infty}
    \sum_{k=1}^{\infty}(a_{k,\ell}^{(\s)}+(-1)^{k+\ell}
    a_{k,\ell}^{(\s')})(n\hbar)^k\hbar^{\ell}\right)\\
  &=
  \Phi_{0}^{(\s)}(\hbar)\Phi_{0}^{(\s')}(-\hbar) e^{\frac{1}{2}\b n^2\hbar}
  \exp\left(\sum_{\ell=-1}^{\infty}
    \sum_{k=1}^{\infty} \b_{k,\ell} (n\hbar)^k\hbar^{\ell}\right) \\
  &= \Phi_{0}^{(\s)}(\hbar)\Phi_{0}^{(\s')}(-\hbar) e^{\frac{1}{2}\b n^2\hbar}
  I(n,\hbar)\,,
\end{aligned}
\ee
where $I(n,\hbar)=\exp\left(\sum_{\ell=-1}^{\infty}
  \sum_{k=1}^{\infty} \b_{k,\ell} (n\hbar)^k\hbar^{\ell}\right)$. 
The coefficient of $\hbar^k$ in $I(n,\hbar)$ 
is a polynomial in $n$ of degree $4k/3$ (resp., $4k/3-4/3$, $4k/3-2/3$) if
$k=0 \bmod 3$ (resp., $k=1 \bmod 3$, $k=2 \bmod 3$), with the first few terms given
by
\be
\label{Inh}
\begin{aligned}
  I(n,\hbar) =&\; 1 + n^2 \b_{2, 0} \hbar^2 + (n^2 \b_{2, 1} + n^4 \b_{4, -1}) \hbar^3
  + (\frac{1}{2} n^4 \b_{2, 0}^2 + n^2 \b_{2, 2} + n^4 \b_{4, 0} ) \hbar^4 \\ & +
  (n^4 \b_{2, 0} \b_{2, 1} + n^2 \b_{2, 3} + n^6 \b_{2, 0} \b_{4, -1} + 
  n^4 \b_{4, 1} + n^6 \b_{6, -1}) \hbar^5 \\ & +
  (\frac{1}{6} n^6 \b_{2, 0}^3 + \frac{1}{2} n^4 \b_{2, 1}^2
  + n^4 \b_{2, 0} \b_{2, 2} + n^2 \b_{2, 4} + n^6 \b_{2, 1} \b_{4, -1} \\ &
  + \frac{1}{2} n^8 \b_{4, -1}^2 + 
 n^6 \b_{2, 0} \b_{4, 0} + n^4 \b_{4, 2} + n^6 \b_{6, 0}) \hbar^6 + O(\hbar^7) \,.
\end{aligned}
\ee
To sum over $n \in \BZ$, we can use either Poisson summation, or the modular
transformation property
\be
\label{thetaS}
\Theta(\e(\tau))
=
\z_8\tau^{-1/2} \,
\Theta(\e(-1/\tau))\,,
\ee
of the theta function
\be
\label{thetaf}
\Theta(q)=\sum_{n\in\BZ}q^{n^2/2} \,,
\ee
where $\z_8=\e(1/8)$ so that $(\z_8)^8=1$. Differentiating
Equation~\eqref{thetaS}, we find that
\be
  \sum_{n\in\BZ}n^{2k}\e\left(\frac{1}{2}n^2\tau\right) \in\z_8
\sum_{j=0}^{k-1}\frac{1}{\tau^{k+j+1/2}}
\BQ\left[\left[\e\left(-\frac{1}{2\tau}\right)\right]\right]
\ee
is polynomial in $\tau^{-1/2}$ with coefficients a power series of $\e(-1/\tau)$.
When $\tau\sim 0$, ignoring all the powers $\e(-1/\tau)$, we obtain that
$\sum_{n\in\BZ}n^{k}\e\left(\frac{1}{2}n^2\tau\right)
\sim \langle n^k \rangle_{2\pi i\tau}$ to all orders in $\tau$, where
\be
\label{bracket}
\langle n^{2k} \rangle_{2\pi i\tau} = 
\frac{(-1)^k\z_8(2k-1)!!}{(2\pi i)^k\tau^{(2k+1)/2}}, \qquad 
\langle n^{2k+1} \rangle_{2\pi i\tau} =0
\ee
and $(2k-1)!!=(2k)!/(k! 2^k)$ for $k \geq 0$. The above equation can also
be written in terms of a formal Gaussian integration as follows:
\be
\label{gaussian}
  \langle n^{2k} \rangle_{2\pi i\tau}
  =
  \int_{\BR}u^{2k}\e\left(\frac{1}{2}u^2\tau\right)du \,.
\ee

Using the above asymptotics with
$2\pi i\tau=\b\hbar$ and Equation~\eqref{Inh} we find that
\be
\label{sumPhi}
\begin{aligned}
  \sum_{n\in\BZ}\Phi_{n}^{(\s)}(\hbar)\Phi_{n}^{(\s')}(-\hbar)
  &=
  \Phi_{0}^{(\s)}(\hbar)\Phi_{0}^{(\s')}(-\hbar)\;
  \langle 1\rangle_{\b\hbar}
  \times\Big(1+\big(-\frac{\b_{2,0}}{\b} +\frac{3 \b_{4,-1}}{\b^2} \big) \hbar
    \Big. \\
    & \hspace{-2cm} \left.
      + \big( 3 \frac{\b_{2, 0}^2}{2 \b^2} - \frac{\b_{2, 1}}{\b}
      - 15 \frac{\b_{2, 0} \b_{4, -1}}{\b^3}
      + 105 \frac{\b_{4, -1}^2}{2 \b^4} + 3 \frac{\b_{4, 0}}{\b^2}
      - 15 \frac{\b_{6, -1}}{\b^3} \big) \hbar^2 + O(\hbar^3)\right) \,.
\end{aligned}
\ee
From the bounds of the $n$-degree of each coefficient of $\hbar^k$ in $I(n,\hbar)$,
it follows that the coefficient of $\hbar^k$ in the sum over $n$ depends only on
$I(n,\hbar)+O(\hbar^{3k+1})$.


\section{Further discussion}
\label{sec.further}

In this section, we discuss briefly three further aspects of the 3D-index: its
relation with Stokes matrices, following~\cite{GGM, GGM:peacock}, its relation
to stability, following~\cite{GL:nahm}, and its possible relation with topological
recursion~\cite{BE:all-order}.

\subsection{The 3D-index and Stokes matrices}
\label{sub.stokes}

A discovery of~\cite{GGM, GGM:peacock} was that the rotated 3D-index essentially
determines the Stokes constants of the Borel resummation of the asymptotic series
that appear in complex Chern--Simons theory. For a precise statement, see
Conjecture 4 of~\cite{GGM:peacock}. A further extension that includes the so-called
trivial flat connection was proposed in~\cite{GGMW:trivial}. 
An explanation of this relation between the BPS counts of the 3D-index and the
Stokes constants of perturbative power series from first principles is currently 
missing. Without doubt, the Stokes constants of the Borel resummation of the matrix
$\Phi_n(h)$ of factorially divergent $h$-power series can be expressed in terms of
the rotated 3D-index, in a manner analogous to the results of~\cite{GGM:peacock},
only that the descendant variable $m$ there is replaced by the variable $n$ here.
We will not pursue this numerical investigation further. 

\subsection{Stability}
\label{sub.stability}

In this short subsection, which is independent of our asymptotic results, we make
some comments regarding the $q=0$ expansion of the colored holomorphic blocks and
their relation with the 3D-index. Whereas the 3D-index is a well-defined topological
invariant, computable from the Neumann--Zagier matrices of an ideal triangulation
of a 3-manifold, the $\BZ \times r$ matrix $H(q)$ of the colored holomorphic blocks
has not been defined in terms of an ideal triangulation, nor is a priori a topological
invariant. Yet, as we will see in examples, the matrix of colored holomorphic blocks
satisfies some stability properties and this makes possible to recover it from the
matrix of the rotated 3D-index.

Perhaps stability properties of $\Irot(n,n')(q)$ as $n$ or $n'$ are large allow
one to determine the colored holomorphic blocks $H(q)$ (in a suitable normalisation)
from the rotated 3D-index.
If so, it would give a definition of the colored holomorphic blocks in terms of
Neumann--Zagier matrices as well as conclude their topological invariance,
something which is currently missing. This works for $4_1$ (whose details we
give) and $5_2$ (whose details we omit). Unfortunately, this hope has not been
realised so far for all knots.

\subsection{Periods from Topological Recursion?}
\label{sub.TR}

We have seen in the introduction that periods of the $A$-polynomial curve
appear in the vertical asymptotics of the meromorphic 3D-index (but not of the
rotated, or the original 3D-index). On the other hand, the rotated 3D-index
satisfies (in two variables) a linear $q$-difference equation whose $q=1$ limit
is conjectured to be the $A$-polynomial of~\cite{CCGLS}. 

Since periods are nonperturbative (and classical) information of the $A$-polynomial
curve, is there an explanation of their appearance in the context of
topological recursion? Note that the latter can explain the asymptotics of the
Kashaev invariant of a knot~\cite{BE:all-order}.


\section{The 3D-index of the $4_1$ knot}
\label{sec.41}


In this section, we give a detailed and computable description of the 3D-index
of the $4_1$ knot and of its asymptotics, illustrating the conjectures of
Section~\ref{sub.conj}. All quantities (functions and numbers) in this section
refer to the $4_1$ knot, which will be omitted from the notation.

\subsection{Holomorphic blocks}

The $4_1$ knots has two colored holomorphic blocks $B^{(1)}(x;q)$ and $B^{(2)}(x;q)$
are given by (see~\cite[Eqn.6.45]{holo-blocks} and~\cite[Sec.5.2]{GGM:peacock})
\be
\begin{aligned}
B^{(1)}(x;q) &=
\frac{(qx^2;q)_{\infty}}{\theta(x;q)\theta(-q^{1/2}x;q)}
\sum_{k=0}^{\infty}(-1)^{k}\frac{q^{k(k+1)/2}x^k}{(q;q)_{k}(qx^2;q)_{k}}\,,\\
B^{(2)}(x;q) &= B^{(1)}(x^{-1};q)\,,\\
B^{(1)}(x;q^{-1}) &=
\frac{\theta(x;q)\theta(-q^{-1/2}x;q)}{(x^2;q)_{\infty}}
\sum_{k=0}^{\infty}(-1)^{k}\frac{q^{k(k+1)/2}x^{-k}}{(q;q)_{k}(qx^{-2};q)_{k}}\,,\\
B^{(2)}(x;q^{-1}) &= B^{(1)}(x^{-1};q^{-1})\,,
\end{aligned}
\ee
where
\be
\label{theta}
  \theta(x;q) =(-q^{1/2}x;q)_{\infty}(-q^{1/2}x^{-1};q)_{\infty}\,.
\ee
Since the two holomorphic blocks are given by 1-dimensional proper $q$-hypergeometric
sums, it follows by Wilf-Zeilberger theory~\cite{WZ}, implemented by Koutschan
in~\cite{Koutschan}, that they both satisfy the second order linear $q$-difference
equation
\be
\label{41.Ahat}
\begin{aligned}
P_{0}(x,q)B^{(\a)}(x;q)+P_{1}(x,q)B^{(\a)}(qx;q)+P_{2}(x,q)B^{(\a)}(q^2x;q)&=0\,,
\end{aligned}
\ee
for $\a=1,2$, where
\be
\begin{aligned}
  P_{0}(x,q)&=q^{2}x^2(q^{3}x^2-1)\,,\\
  P_{1}(x,q)&=-q^{1/2}(1-q^{2}x^2)(1-qx-qx^2-q^{3}x^2-q^{3}x^3+q^{4}x^4)\,,\\
  P_{2}(x,q)&=q^{3}x^2(-1+qx^2) \,.
\end{aligned}
\ee
Our normalisation of the above $q$-difference equation differs slightly
from~\cite[Eqn.(140)]{GGM:peacock}. Our choice was dictated by the facts that it
annihilates in two ways the rotated 3D-index, the latter being canonically normalised,
and by the fact that it respects the $\BZ/2\BZ$ Weyl symmetry,
(i.e., it is invariant under the map $x \mapsto 1/x$), which when $x=q^n$,
($n$ being the weight of an $\mathfrak{sl}_2(\BC)$ representation), means invariance
under the map $n \mapsto -n$. 

\subsection{Colored holomorphic blocks}
\label{sub.41h}


The holomorphic blocks are meromorphic functions of $(x,q)$. We now define
the colored holomorphic blocks by looking at their expansions at $x=q^n e^u$ at
$u=0$. Explicitly, we define the colored holomorphic blocks $\hb^{(0)}_{n}(q)$
and $\hb^{(1)}_{n}(q)$ for $|q|<1$ by

\be
\label{41h}
\begin{aligned}
  B^{(1)}(q^ne^u;q) &=\frac{1}{(q;q)_{\infty}(-q^{1/2};q)_{\infty}^{2}}
  \hb^{(0)}_n(q)u^{-1}+O(u^0)\,,\\
  B^{(1)}(q^ne^u;q)+B^{(2)}(q^ne^u;q)&=
  \frac{1}{(q;q)_{\infty}(-q^{1/2};q)_{\infty}^{2}}\hb^{(1)}_{n}(q)+O(u^1) \,,
\end{aligned}  
\ee
and for $|q|<1$ by
\be
\label{41hb}
\begin{aligned}
  B^{(1)}(q^{-n}e^{-u};q^{-1})&=
  \frac{(q;q)_{\infty}(-q^{1/2};q)_{\infty}^{2}}{2}\hb^{(0)}_n(q^{-1})+O(u^1)\,, \\
  B^{(1)}(q^{-n}e^{-u};q^{-1})-B^{(2)}(q^{-n}e^{-u};q^{-1}) &=
  \frac{(q;q)_{\infty}(-q^{1/2};q)_{\infty}^{2}}{2}\hb^{(1)}_{n}(q^{-1})u+O(u^2)\,.
\end{aligned}
\ee
Since the holomorphic blocks $B^{(\a)}(x;q)$ satisfy the $q$-difference
equation~\eqref{41.Ahat}, it follows that the colored holomorphic blocks, defined
in~\eqref{41h} and~\eqref{41hb} satisfy the linear $q$-difference equation
\be
\label{41.Ahath}
P_{0}(q^n,q)\hb^{(\a)}_n(q)+P_{1}(q^n,q)\hb^{(\a)}_{n+1}(q)
+P_{2}(q^n,q)\hb^{(\a)}_{n+2}(q)=0 \qquad (\a=0,1, \, \, n \in \BZ) \,.
\ee

Next, we give $q$-hypergeometric formulas for the colored holomorphic blocks
for $|q| \neq 1$. To do so, we use the next lemma which can be deduced
from~\cite[Prop.2.2]{GK:qseries} or from~\cite{GZ:qseries}.
For a positive integer $\ell$, consider the function
\be
\label{Eell}
\begin{aligned}
  E_{\ell}(q)
  =
  \frac{\zeta(1-\ell)}{2}+\sum_{s=1}^{\infty}s^{\ell-1}\frac{q^{s}}{1-q^s}\,,
\end{aligned}
\ee
(where $\zeta(s)$ is the Riemann zeta function), analytic for $|q|<1$ and
extended to $|q|>1$ by the symmetry $E_{\ell}(q^{-1}) = -E_{\ell}(q)$. 

\begin{lemma}[\cite{GK:qseries,GZ:qseries}]
\label{lem.epexpansions}
We have:  
\be
\label{epasy}
\begin{small}
\begin{aligned}
\frac{(qe^{\ve};q)_{\infty}}{(q;q)_{\infty}}\frac{1}{(qe^{\ve};q)_{m}}=
&\frac{1}{(q;q)_{m}}\sqrt{\frac{-\ve}{1-\exp(\ve)}}
\exp\left(-\sum_{\ell=1}^{\infty}
  \left(E_{\ell}(q)-\sum_{n=1}^{m}\Li_{1-\ell}(q^{n})\right)
  \frac{\ve^\ell}{\ell!}\right)\,, \\
\frac{(q^{-1}e^{\ve};q^{-1})_{\infty}}{(q^{-1};q^{-1})_{\infty}}
\frac{1}{(q^{-1}e^{\ve};q^{-1})_{m}}=
&
\frac{1}{(q^{-1};q^{-1})_{m}}\frac{-1}{\ve}\sqrt{\frac{-\ve}{1-\exp(\ve)}} \\
& \exp\left(-\sum_{\ell=1}^{\infty}
  \left(E_{\ell}(q^{-1})-\sum_{n=1}^{m}\Li_{1-\ell}(q^{-n})\right)
  \frac{\ve^{\ell}}{\ell!}\right) \,.
\end{aligned}
\end{small}
\ee
\end{lemma}

\begin{proposition}
  \label{prop.41h}
  We have: 
\be
\label{41h0}
\hb_{n}^{(0)}(q) = (-1)^nq^{|n|(2|n|+1)/2}
\sum_{k=0}^{\infty}(-1)^{k}\frac{q^{k(k+1)/2+|n|k}}{(q;q)_{k}(q;q)_{k+2|n|}}\,,
\ee
and
\be
\label{41h1}
\begin{small}
\begin{aligned}
\hb^{(1)}_{n}(q)
&=
(-1)^n q^{|n|(2|n|+1)/2}\sum_{k=0}^{\infty}\left(-4E_{1}(q)+
  \sum_{\ell=1}^{k+2|n|}\frac{1+q^\ell}{1-q^\ell}
  +\sum_{\ell=1}^{k}\frac{1+q^\ell}{1-q^\ell}\right)
(-1)^k\frac{q^{k(k+1)/2+|n|k}}{(q;q)_k(q,q)_{k+2|n|}}\\
&\quad-2(-1)^n q^{|n|(2|n|-1)/2}\sum_{k=0}^{2|n|-1}(-1)^k\frac{q^{k(k+1)/2-|n|k}
(q^{-1},q^{-1})_{2|n|-1-k}}{(q;q)_{k}}\,,
\end{aligned}
\end{small}
\ee
for $|q| \neq 1$.
The colored holomorphic blocks satisfy the symmetries
\be
\label{41hsym1}
\hb_{n}^{(0)}(q^{-1})=\hb_{n}^{(0)}(q), \qquad \hb_{n}^{(1)}(q^{-1})=-\hb_{n}^{(1)}(q) \,,
\ee
and
\be
\label{41hsym2}
\hb_{-n}^{(\a)}(q) = \hb_{n}^{(\a)}(q), \qquad \a=0,1 \,. 
\ee
\end{proposition}

\begin{proof}
The proof of the above formulas~\eqref{41h0} and~\eqref{41h1} follows from the
definitions of the colored holomorphic blocks combined with the expansions
of Lemma~\ref{lem.epexpansions} above. One simply collects terms associated to
each piece of the sum to deduce the proposition. For example,
\be
  \frac{1}{\theta(q^{n}e^{u};q)}
  =
  \frac{q^{n^{2}/2}}{(-q^{1/2};q)_{\infty}^2}(1+nu+O(u^2))\,,
\ee
or when $n\geq0$,
\be
\frac{(q^{2n+1}x^2;q)_{\infty}}{(q^{2n+1}x^2;q)_{k}}
=\frac{(q;q)_{\infty}}{(q;q)_{k+2n}}\left(1-\left(\frac{1}{2}+2E_{1}(q)
+2\sum_{\ell=1}^{k+2n} \frac{q^{\ell}}{1-q^\ell}\right)u+O(u^2)\right)\,,
\ee
or when $k-2n<0$,
\be
\frac{(q^{-2n+1}x^{-2};q)_{\infty}}{(q^{-2n+1}x^{-2};q)_{k}}
= -2(q^{-1};q^{-1})_{2n-k-1}u+O(u^2)\,.
\ee
The $q \leftrightarrow q^{-1}$ symmetry of $\hb_{n}^{(0)}(q)$ follows from the
definition and the symmetry of its summand itself and similarly for $h^{(1)}(q)$ using 
the first few coefficients in the $u$ expansion, where we also use the definition of
$E_{1}(q^{-1})$.
\end{proof}

The colored holomorphic blocks when $n=0$ agree with the pair of $q$-series
in~\cite{GZ:qseries}:
\be
\label{41gG}
\hb_{0}^{(0)}(q)=g(q), \qquad \hb_{0}^{(1)}(q)=G(q) \,.
\ee
This is manifest from the definitions of these $q$-series.
The relation to the descendant $q$-series of~\cite[Eqn.(13)]{GGM} is given by
\be
\begin{aligned}
  g^{(i)}_0(q)&=\hb^{(i)}_{0}(q)\,,\\
  (1-q)g^{(i)}_1(q)&=\hb^{(i)}_{0}(q)-q^{1/2}\hb^{(i)}_{1}(q)\,.
\end{aligned}
\ee
This follows from the relation for the two Wronskians~\cite[Eqn.(156)]{GGM:peacock}
noting the $-q^{1/2}$ difference as seen in Equation~\eqref{41.Ahat}. Note that the
colored Jones polynomial of the $4_1$ knot is given
\be
\label{41cj}
\begin{aligned}
  J_{n}(q)
  &=
  \sum_{k=0}^{n-1}q^{-kn}(q^{n-1};q^{-1})_{k}(q^{n+1};q)_{k}\\
  &=
  \frac{1}{1-q^n}\sum_{k=-n}^{-1}q^{nk+n}\frac{(q;q)_{n-k-1}}{(q;q)_{k+n}}\\
  &=
  \frac{1}{1-q^n}\sum_{k=1}^{n}q^{n-nk}\frac{(q;q)_{n+k-1}}{(q;q)_{n-k}}\,.
\end{aligned}
\ee
The last formula appears up to some simple factor in the expression for
$\hb^{(1)}_{n}(q)$. The first few values of the colored holomorphic blocks
are given by
\be
\begin{aligned}
\hb^{(0)}_{0}(q) &= 1 - 2 - 2q^2 - 2q^3 - 2q^4 + q^6 + 5q^7 + 7q^8 + 11q^9 + \dots \,, \\
\hb^{(0)}_{1}(q) &= q^{1/2}(-q - q^2 - q^3 + q^5 + 3q^6 + 5q^7 + 7q^8 + 9q^9+\dots) \,,\\
\hb^{(0)}_{2}(q) &= q^5 + q^6 + 2q^7 + 2q^8 + 3q^9 + \dots\,,
\end{aligned}
\ee
and
\be
\begin{aligned}
  \hb^{(1)}_{0}(q) &= 1 - 7q - 14q^2 - 8q^3 - 2q^4 + 30q^5 + 43q^6 + 95q^7
  + 109q^8 + 137q^9 + \dots\,,\\
  \hb^{(1)}_{1}(q) &= q^{1/2}(-2q^{-1} - 5q - 3q^2 + 3q^3 + 16q^4 + 33q^5 + 51q^6
  + 73q^7 + 77q^8 + 81q^9 +\dots)\,,\\
  \hb^{(1)}_{2}(q) &= 2q^{-3} - 2q^{-2} + 2q + 2q^2 + 2q^3 + 2q^4 + 11q^5 + 9q^6
  + 14q^7 + 6q^8 + 3q^9 + \dots\,.
\end{aligned}
\ee
The colored holomorphic blocks have the $q$-degree (i.e., minimum power of $q$)
\be
\label{41deg}
\deg_q \hb^{(0)}_{n}(q)=n^2 + \frac{1}{2}|n|,
\qquad \hb^{(1)}_{n}(q)=-n^2 + \frac{1}{2}|n| \,.
\ee
The above statement for the degree follows easily from Equations~\eqref{41h0}
and~\eqref{41h1}. Alternatively, observe that the two colored holomorphic blocks
are fundamental solutions of the linear $q$-difference equation~\eqref{41.Ahath}
obtained by the Frobenius (or the WKB) method (one power of $q^n$ at a time,
as explained in~\cite{Ga:quasi}), and their degrees are quadratic polynomials
whose slopes are $4$ and $-4$, those of the Newton polygon of the $A$-polynomial. 
  
\subsection{The rotated and the meromorphic 3D-index}
\label{cub.41rotmer}

We next express the rotated 3D-index in terms of the colored holomorphic blocks.

\begin{proposition}
  \label{prop.41Irot}
  For all integers $n$ and $n'$ we have:
  \be
  \label{41.irot}
  \Irot(n,n')(q)
  =
  -\frac{1}{2}\hb_{n'}^{(1)}(q^{-1})\hb_{n}^{(0)}(q)
  +\frac{1}{2}\hb_{n'}^{(0)}(q^{-1})\hb_{n}^{(1)}(q)\,.
\ee
\end{proposition}

\begin{proof}
  Equation~\eqref{IrotB}, and the definition of the colored holomorphic blocks
  imply that 
\be
\begin{aligned}
  \Irot(n,n')(q) &= \lim_{u \to 0}(B^{(1)}(q^{-n'}e^{-u};q^{-1})B^{(1)}(q^{n}e^u;q)
  +B^{(2)}(q^{-n'}e^{-u};q^{-1})B^{(2)}(q^{n}e^u;q))\\
  &=\lim_{u \to 0}(
  -\frac{1}{2}\hb_{n'}^{(1)}(q^{-1})\hb_{n}^{(0)}(q)
  +\frac{1}{2}\hb_{n'}^{(0)}(q^{-1})\hb_{n}^{(1)}(q)+O(u^1))\\
  &=-\frac{1}{2}\hb_{n'}^{(1)}(q^{-1})\hb_{n}^{(0)}(q)
  +\frac{1}{2}\hb_{n'}^{(0)}(q^{-1})\hb_{n}^{(1)}(q) \,.
\end{aligned}
\ee
This and Equation~\eqref{ImerIrot} imply~\eqref{41.mer}. Note that the sum
over the integers in~\eqref{41.mer} is a convergent $q$-series since
$\deg_q \hb_{n}^{(1)}(q)\hb_{n}^{(0)}(q)=|n|$ (as follows from~\eqref{41deg}). 
\end{proof}

Using Equations~\eqref{41h0} and~\eqref{41h1} and the above proposition, it follows
that the first few values of $\Irot(n,n)(q)=\hb_n^{(0)}(q) \hb_n^{(1)}(q)$
are given by 
\be
\begin{aligned}
  \Irot(0,0)(q) &= 1 - 8q - 9q^2 + 18q^3 + 46q^4 + 90q^5 + 62q^6
  + 10q^7 - 170q^8 - 424q^9 + \dots\,,\\
  \Irot(1,1)(q) &= 2q + 2q^2 + 7q^3 + 8q^4 + 3q^5 - 22q^6 - 67q^7
  - 132q^8 - 206q^9 +\dots\,,\\
  \Irot(2,2)(q) &= 2q^2 + 2q^4 + 4q^6 + 2q^7 + 8q^8 + 8q^9 + \dots\,,
\end{aligned}
\ee
and their degree is given by $\deg_q \Irot(n,n)(q)=|n|$. This, together with 
Equation~\eqref{ImerIrot}, gives the first few terms of the meromorphic 3D-index
as stated in Equation~\eqref{41.mer}. Incidentally, the diagonal $\Irot(n,n)(q)$
of the rotated 3D-index satisfies the following third order recursion
\be
\begin{tiny}
\begin{aligned}
  0=&-q^{4 + 4n}(-1 + q^{2 + n})(1 + q^{2 + n})(-1 + q^{5 + 2n})
   (1 - q^{2 + n} - q^{3 + 2n} - q^{5 + 2n} - q^{6 + 3n} + q^{8 + 4n})
   \Irot(n,n)(q)\\
  &+ q(-1 + q^{1 + n})(1 + q^{1 + n})(-1 + q^{5 + 2n})
  (1 - q^{1 + n} - q^{1 + 2n} - q^{3 + 2n} - q^{3 + 3n} + q^{4 + 4n})\\
  &\qquad\times
  (1 - q^{1 + n} - q^{2 + n} - q^{1 + 2n} - q^{2 + 2n} + q^{3 + 2n} - 
   q^{4 + 2n} - q^{5 + 2n} + q^{3 + 3n} + q^{6 + 3n} + q^{3 + 4n} + 
   q^{5 + 4n}\\
  &\qquad\qquad + 3q^{6 + 4n} + q^{7 + 4n} + q^{9 + 4n} + q^{6 + 5n} + 
   q^{9 + 5n} - q^{7 + 6n} - q^{8 + 6n} + q^{9 + 6n}\\
  &\qquad\qquad - q^{10 + 6n} - 
   q^{11 + 6n} - q^{10 + 7n} - q^{11 + 7n} + q^{12 + 8n})
  \Irot(n+1,n+1)(q)\\
  &- (-1 + q^{2 + n})(1 + q^{2 + n})(-1 + q^{1 + 2n})
  (1 - q^{2 + n} - q^{3 + 2n} - q^{5 + 2n} - q^{6 + 3n} + q^{8 + 4n})\\
  &\qquad\times
  (1 - q^{1 + n} - q^{2 + n} - q^{1 + 2n} - q^{2 + 2n} + q^{3 + 2n} - 
   q^{4 + 2n} - q^{5 + 2n} + q^{3 + 3n} + q^{6 + 3n} + q^{3 + 4n} + 
   q^{5 + 4n}\\
  &\qquad\qquad + 3q^{6 + 4n} + q^{7 + 4n} + q^{9 + 4n} + q^{6 + 5n} + 
   q^{9 + 5n} - q^{7 + 6n} - q^{8 + 6n} + q^{9 + 6n} - q^{10 + 6n} - 
   q^{11 + 6n} - q^{10 + 7n}\\
  &\qquad\qquad - q^{11 + 7n} + q^{12 + 8n})
  \Irot(n+2,n+2)(q)\\
  &+ q^{9 + 4n}(-1 + q^{1 + n})(1 + q^{1 + n})
  (-1 + q^{1 + 2n})(1 - q^{1 + n} - q^{1 + 2n} - q^{3 + 2n} - 
   q^{3 + 3n} + q^{4 + 4n})\Irot(n+3,n+3)(q)\,,
\end{aligned}
\end{tiny}
\ee
for all integers $n$.

\subsection{Stability}
\label{sub.41stability}

In this section, we discuss briefly the stability properties of the colored holomorphic
blocks of the $4_1$ knot. By stability, we mean the limit as $n \to \infty$
of a sequence of $q$-series, as explained in detail in~\cite{GL:nahm}. 
Equations~\eqref{41h0} and~\eqref{41h1} imply easily that 
\be
\begin{aligned}
\lim_{n\rightarrow\infty}\hb^{(0)}_{n}(q)q^{-n(2n+1)/2}&=\frac{1}{(q;q)_{\infty}} \,,\\
\lim_{n\rightarrow\infty}\hb^{(1)}_{n}(q)q^{n(2n-1)/2}&=2(q;q)_{\infty}\,.
\end{aligned}
\ee
It follows that 
\be
\begin{aligned}
  \lim_{n'\rightarrow\infty}q^{n'(2n'+1)/2}\Irot(n,n')(q)
  &=
  \lim_{n'\rightarrow\infty}\frac{q^{n'(2n'+1)/2}}{2}\hb_{n'}^{(1)}(q)\hb_{n}^{(0)}(q)
  +\frac{q^{n'(2n'+1)/2}}{2}\hb_{n'}^{(0)}(q)\hb_{n}^{(1)}(q)\\
  &=(q;q)_{\infty}h_{n}^{(0)}(q)\,,
\end{aligned}
\ee
which implies that
\be
  \frac{\Irot(n,n)(q)}{(q;q)_{\infty}h_{n}^{(0)}(q)}
  =
  \frac{h_{n}^{(1)}(q)}{(q;q)_{\infty}}\,.
\ee
In other words, the degree~\eqref{41deg} of the colored holomorphic blocks
$h^{(\a)}_n(q)$ for $\a=0,1$ and the rotated 3D-index determines thm up to
multiplication by a power of $(q;q)_\infty$. 

\subsection{Numerical asymptotics of the meromorphic 3D-index}
\label{sub.41numimer}

In this section, we discuss the numerical computation of the asymptotics of the
meromorphic 3D-index at $q\to 1$ following closely the ideas and the numerical methods
of~\cite{GZ:kashaev,GZ:qseries}. We first fix a ray and use $\tau$ in that ray with
absolute value $|\tau|=1/N$ for $N=200,\dots,300$. For these values of $\tau$, we
compute numerically the values of the colored holomorphic blocks
$h^{(\a)}_n(q)$ for $n=0,1$ and $\a=0,1$ with precision about
20000 digits using their $q$-hypergeometric definition~\eqref{41h0} and~\eqref{41h1}.
We then use the $q$-difference equation~\eqref{41.Ahath} to compute the values of the
colored holomorphic blocks $h^{(\a)}_n(q)$ for $n=0,\dots,25N$ and $\a=0,1$, and sum
up these terms using~\eqref{41.mer}, keeping in mind that the error is $O(q^N)$. 
Having done so, we obtain the values of the meromorphic 3D-index at the above range
of $\tau$ to about 30 digits. We then numerically extrapolate the coefficients of each
power of $\tau$ in the asymptotic series, and finally recognise them exactly.
After doing so, we find numerically the following asymptotic expansion given in
Equation~\eqref{Imernum}, as $\arg(\tau)$ is fixed, positive and small, and
$\tau$ tends to zero.

\subsection{Asymptotic series}
\label{sub.41WKB}

In this section, we will explain how to compute the asymptotic series
$\Phi^{(\a)}(\hbar)$ using the discrete WKB applied to the $\Ahat$-equation of the
figure-eight knot.

\begin{proposition}
  If the ansatz~\eqref{Ahat.ansatz} satisfies
  \be
  \label{Ahat.41.WKB}
  P_{0}(e^{n\hbar},e^\hbar)\Phi^{(\s)}_n(\hbar)
  +P_{1}(e^{n\hbar},e^\hbar)\Phi^{(\s)}_{n+1}(\hbar)
  +P_{2}(e^{n\hbar},e^\hbar)\Phi^{(\s)}_{n+2}(\hbar)=0\,,
  \ee
  it follows that $a_{2,-1}$ is a root of the polynomial equation
  $4 a_{2,-1}^2+3=0$. 
  A choice of such a root and the initial condition at $n=0$
  uniquely determines the power series $\Phi^{(\s)}_{n}(h)$\,.
\end{proposition}

The first few terms of the asymptotic series are given by
\be
\label{41Phi7terms}
\begin{tiny}
\begin{aligned}
\widehat{\Phi}^{(\s_1)}_{n}(\hbar)
&=
(-3)^{-1/4} e^{\frac{i\mathrm{Vol}}{\hbar}}
\exp\Big(\sqrt{-3}\hbar\frac{n^2}{2}\Big)
\Big(1-\frac{11}{24\sqrt{-3}^3}\hbar\\
&\quad + \Big(-\frac{9}{2\sqrt{-3}^6}n^2 + \frac{697}{1152\sqrt{-3}^6}\Big)\hbar^2\\
&\quad + \Big(-\frac{81}{4\sqrt{-3}^9}n^4 + \frac{285}{16\sqrt{-3}^9}n^2
- \frac{724351}{414720\sqrt{-3}^9}\Big)\hbar^3\\
&\quad+ \Big(\frac{4185}{32\sqrt{-3}^{12}}n^4
- \frac{19825}{256\sqrt{-3}^{12}}n^2
+ \frac{278392949}{39813120\sqrt{-3}^{12}}\Big)\hbar^4\\
&\quad + \Big(\frac{12879}{20\sqrt{-3}^{15}}n^6
- \frac{129933}{64\sqrt{-3}^{15}}n^4 + \frac{5481733}{5760\sqrt{-3}^{15}}n^2
- \frac{17049329117}{209018880\sqrt{-3}^{15}}\Big)\hbar^5\\
&\quad + \Big(\frac{6561}{8\sqrt{-3}^{18}}n^8 - \frac{646461}{80\sqrt{-3}^{18}}n^6
+ \frac{18984649}{1280\sqrt{-3}^{18}}n^4
- \frac{1718762189}{276480\sqrt{-3}^{18}}n^2
+ \frac{39021801798779}{75246796800\sqrt{-3}^{18}}\Big)\hbar^6+O(\hbar^7)\Big) \,.
\end{aligned}
\end{tiny}
\ee
We have computed 300 terms of the series $\widehat{\Phi}^{(\s_1)}_{0}(\hbar)$,
obtained either by extrapolation, and matching the stationary phase of the
corresponding state-integral. Using 10 terms of the above series, one can compute
the series $\widehat{\Phi}^{(\s_1)}_{n}(\hbar)$ up to $O(\hbar^{11})$.

\begin{proof}
To determine $a_{2,-1}$ (which satisfies the same polynomial equation
as the cusp shape) note that
\be
\begin{aligned}
P_0(1+u,1) &= u^4 + 4u^3 + 5u^2 + 2u\,,\\
P_1(1+u,1) &= u^6 + 5u^5 + 7u^4 - 2u^3 - 10u^2 - 4u\,,\\
P_2(1+u,1) &= u^4 + 4u^3 + 5u^2 + 2u\,,
\end{aligned}
\ee
and so
\be
P_{0}(1+u,1)
+P_{1}(1+u,1)(1+2a_{2,-1}u)
+P_{2}(1+u,1)(1+2a_{2,-1}u)^2
=
(8a_{2,-1}^2 + 6)u^3+O(u^4)\,.
\ee
Therefore, $a_{2,-1}=\frac{1}{2}\sqrt{-3}$. Choosing a root, substituting
into~\eqref{Ahat.ansatz}, and then expanding~\eqref{Ahat.41.WKB} to $O(\hbar^6)$
we find
\be
\begin{aligned}
&\Big(\big(16a_{4,-1} + 8\sqrt{-3}a_{2,0} - \frac{8}{3}\sqrt{-3}\big)n^3
+ \dots\Big)\hbar^4\\
&\qquad + \Big(\big(24\sqrt{-3}a_{4,-1} - 12a_{2,0} + 8\big)n^5
+\dots\Big)\hbar^5+O(\hbar^6)=0\,,
\end{aligned}
\ee
which give two independent equations for $a_{4,-1}$ and $a_{2,0}$ which have solution
\be
a_{4,-1}=\frac{\sqrt{-3}}{12}, \qquad  a_{2,0}=\frac{1}{6}\,.
\ee
This shows that 
\be
\Phi_{0}(\hbar)^{-1}\Phi_n(\hbar)+O(\hbar^3)
=
1 + \frac{\sqrt{-3}}{2}n^2h + \left(-\frac{3}{8}n^4 + \frac{1}{6}n^2\right)h^2+O(h^3)\,.
\ee
We can continue inductively in $N$ computing $a_{k,\ell}$ where $k+2\ell=2N$. The
coefficients of $n^3\hbar^{3+N},\dots,n^{3+2N}\hbar$ form a linear system in
$a_{k,\ell}$ with $k+2\ell=N$ which depends on $a_{k,\ell}$ with $k+2\ell<N$.
Using the initial condition
\be
\label{41Phi0}
\begin{small}
\begin{aligned}
\widehat{\Phi}^{(\s_1)}_{0}(\hbar) &= 
(-3)^{-1/4} e^{\frac{i\mathrm{Vol}}{\hbar}}
\bigg(1-\frac{11}{24\sqrt{-3}^3}\hbar + \frac{697}{1152\sqrt{-3}^6}\hbar^2
-\frac{724351}{414720\sqrt{-3}^9}\hbar^3\\
&\quad+\frac{278392949}{39813120\sqrt{-3}^{12}}\hbar^4
- \frac{17049329117}{209018880\sqrt{-3}^{15}}\hbar^5
  +\frac{39021801798779}{75246796800\sqrt{-3}^{18}}\hbar^6+O(\hbar^7)\bigg)\,,
\end{aligned}
\end{small}
\ee
taken from~\cite{GZ:kashaev}, together with the above discussion
gives~\eqref{41Phi7terms}.
\end{proof}
Finally, we define the second asymptotic series $\widehat{\Phi}^{(\s_2)}_{n}(\hbar)$ 
by using the other choice of $a_{2,-1}$ in the WKB. It is easy to see that
the two asymptotic series are related by
\be
\label{41Phi12}
\widehat{\Phi}^{(\s_2)}_{n}(\hbar)= i \widehat{\Phi}^{(\s_1)}_{n}(-\hbar) \,,
\ee
from which it follows that they trivially
satisfy the quadratic relation
\be
\label{41Phiquadrel}
\widehat{\Phi}^{(\s_1)}_{n}(\hbar)\widehat{\Phi}^{(\s_1)}_{n}(-\hbar)
+\widehat{\Phi}^{(\s_2)}_{n}(\hbar)\widehat{\Phi}^{(\s_2)}_{n}(-\hbar)=0\,.
\ee

\subsection{Asymptotics of the colored holomorphic blocks}
\label{sub.41hn}

Having defined the asymptotic series $\wh\Phi^{(\s_1)}_n(\hbar)$ and 
$\wh\Phi^{(\s_2)}_n(\hbar)$, we now use them to give the asymptotics of the colored
holomorphic blocks in suitable rays, with exponentially small terms included.
To Borel resum the above two asymptotic series, we use about 300 coefficients
of $\hbar$ when $n=0$, and about 170 coefficients (each a polynomial of $n$) for
general $n$.
Using for example $\tau=\e(0.001)/50$, one can check the following
identities~\eqref{41HPhi} to $10$ to $20$ coefficients of $\tq$ within the accuracy
of the numerical Borel resummation.
  
For $\tau$ in a cone $C = \arg(\th) \in (\th_0,\th_1)$
with $0< \th_0 < \th_1$ sufficiently small, we find numerically that 
\be
\label{41HPhi}
  \hb_{n}^{(0)}(q)=\z_8\tau^{1/2}\sum_{k=1}^{2}
  s \widehat{\Phi}^{(\s_{k})}_{n}(2\pi i\tau)H_{\s_k}^{(0)}(\tq),  \quad 
  \hb_{n}^{(1)}(q)=\z_8\tau^{-1/2}\sum_{k=1}^{2}
  s \widehat{\Phi}^{(\s_{k})}_{n}(2\pi i\tau)H_{\s_k}^{(1)}(\tq)\,,
\ee
where
\be
  H_{\s_1}^{(\a)}(\tq)=h^{(j)}_{0}(\tq), \qquad
  H_{\s_2}^{(\a)}(\tq)=
  \frac{2\tq-1}{1-\tq}h^{(j)}_{0}(\tq)-\frac{\tq^{1/2}}{1-\tq}h^{(j)}_{1}(\tq)\,,
\ee
for $\a=0,1$. The above identities were found in~\cite[Sec.5]{GGM:peacock} when
$n=0$ and were verified for general $n$ here. We stress that Equation~\eqref{41HPhi}
is an identity of analytic functions on the cone $C$. Writing the above identity
in matrix form, we obtain that 
\be
\label{41H}
\begin{aligned}
  \begin{pmatrix}
    \hb_{n}^{(0)}(q) & \hb_{n}^{(1)}(q)\\
    \hb_{n+1}^{(0)}(q) & \hb_{n+1}^{(1)}(q)\\
  \end{pmatrix}
  &=
  \begin{pmatrix}
    s\widehat{\Phi}^{(\s_{1})}_{n}(2\pi i\tau) &
    s\widehat{\Phi}^{(\s_{2})}_{n}(2\pi i\tau)\\
    s\widehat{\Phi}^{(\s_{1})}_{n+1}(2\pi i\tau) &
    s\widehat{\Phi}^{(\s_{2})}_{n+1}(2\pi i\tau)\\
  \end{pmatrix}
  \begin{pmatrix}
    H_{\s_1}^{(0)}(\tq) & H_{\s_1}^{(1)}(\tq)\\
    H_{\s_2}^{(0)}(\tq) & H_{\s_2}^{(1)}(\tq)\\
  \end{pmatrix}
  \begin{pmatrix}
    \tau^{1/2} & 0\\
    0 & \tau^{-1/2}\\
  \end{pmatrix}\,,
\end{aligned}
\ee
where
\be
\label{41H2}
\begin{aligned}
\begin{pmatrix}
    H_{\s_1}^{(0)}(\tq) & H_{\s_1}^{(1)}(\tq)\\
    H_{\s_2}^{(0)}(\tq) & H_{\s_2}^{(1)}(\tq)\\
\end{pmatrix}
&= \begin{pmatrix} 1 & 0 \\ \frac{2-\tq}{1-\tq} & -\frac{\tq^{1/2}}{1-\tq}
\end{pmatrix}
  \begin{pmatrix}
    \hb_{0}^{(0)}(\tq) & \hb_{0}^{(1)}(\tq)\\
    \hb_{1}^{(0)}(\tq) & \hb_{1}^{(1)}(\tq)\\
  \end{pmatrix}
\\
&\hspace{-2cm}=\begin{pmatrix}
    1-\tq-2\tq^2-2\tq^3-2\tq^4+\dots & 1-7\tq-14\tq^2-8\tq^3-2\tq^4+\dots\\
    -1+2\tq+3\tq^2+2\tq^3+\tq^4+\dots & 1+10\tq+15\tq^2+2\tq^3+19\tq^4+\dots
\end{pmatrix} \,.
\end{aligned}
\ee
Although we will not use this in our current paper, we point out that the entries of
\be
  \begin{pmatrix}
    \hb_{n'}^{(0)}(\tq) & \hb_{n'}^{(1)}(\tq)\\
    \hb_{n'+1}^{(0)}(\tq) & \hb_{n'+1}^{(1)}(\tq)\\
  \end{pmatrix}^{-1}
  \begin{pmatrix}
    \tau^{-1/2} & 0\\
    0 & \tau^{1/2}\\
  \end{pmatrix}
  \begin{pmatrix}
    \hb_{n}^{(0)}(q) & \hb_{n}^{(1)}(q)\\
    \hb_{n+1}^{(0)}(q) & \hb_{n+1}^{(1)}(q)\\
  \end{pmatrix}
\ee
can be expressed as combinations of the elementary functions holomorphic for
$\tau\in\BC'$ times state integrals (see~\cite{GK:qseries, GGM:peacock} with
the notation as in those papers)
\be
\int
\Phi_{\sfb}(x+i\sfb^{-1}(m\sfb+m'\sfb^{-1}))
\Phi_{\sfb}(x-i\sfb^{-1}(m\sfb+m'\sfb^{-1}))\exp\left(-\pi i x^{2}\right)dx\,,
\ee
and therefore extend to analytic functions for $\tau\in\BC'$.

Now for $\tau \in C$, the above discussion and Equation~\eqref{41.irot}
imply that the rotated 3D-index
is given by
\be
\label{hnnasy}
\begin{small}
\begin{aligned}
  \Irot(n,n')(q)&= -\frac{1}{2}\hb_{n'}^{(1)}(q^{-1})\hb_{n}^{(0)}(q)
+\frac{1}{2}\hb_{n'}^{(0)}(q^{-1})\hb_{n}^{(1)}(q) \\
  & \hspace{-2cm} =
  \frac{1}{2}\hb_{n'}^{(1)}(q)\hb_{n}^{(0)}(q)
  +\frac{1}{2}\hb_{n'}^{(0)}(q)\hb_{n}^{(1)}(q)\\
  & \hspace{-2cm} =
  \frac{i}{2}(s \widehat{\Phi}^{(\s_{1})}_{n'}(2\pi i\tau)H_{\s_1}^{(1)}(\tq)
  +s \widehat{\Phi}^{(\s_2)}_{n'}(2\pi i\tau)H_{\s_2}^{(1)}(\tq))
  (s \widehat{\Phi}^{(\s_{1})}_{n}(2\pi i\tau)H_{\s_1}^{(0)}(\tq)
  +s \widehat{\Phi}^{(\s_2)}_{n}(2\pi i\tau)H_{\s_2}^{(0)}(\tq))\\
  & \hspace{-2cm} +
  \frac{i}{2}(s \widehat{\Phi}^{(\s_{1})}_{n'}(2\pi i\tau)H_{\s_1}^{(0)}(\tq)
  +s \widehat{\Phi}^{(\s_2)}_{n'}(2\pi i\tau)H_{\s_2}^{(0)}(\tq))
  (s \widehat{\Phi}^{(\s_{1})}_{n}(2\pi i\tau)H_{\s_1}^{(1)}(\tq)
  +\tau^{1/2}s \widehat{\Phi}^{(\s_2)}_{n}(2\pi i\tau)H_{\s_2}^{(1)}(\tq))\,.
\end{aligned}
\end{small}
\ee
Now, let $\tau\to 0$ with $\arg(\tau)$ fixed, positive and sufficiently small.
Then, $\tq$ is exponentially small compared to $\tau$, and ignoring  
the exponential small corrections in the above identity, and using the
constant terms of the series in~\eqref{41H2} and the quadratic
relations~\eqref{41Phiquadrel}
we obtain that
\be
\label{hnnasy2}
\begin{aligned}
  \Irot(n,n)(q)
  &\sim
  \widehat{\Phi}^{(\s_{1})}_{n}(2\pi i\tau)
  \widehat{\Phi}^{(\s_{2})}_{n'}(-2\pi i\tau)\,.
\end{aligned}
\ee
  
\subsection{Asymptotics of the meromorphic 3D-index}
\label{sub.merasy.41}

In the previous sections we gave ``horizontal'' asymptotic expansion of the
colored holomorphic blocks and of the rotated 3D-index. We now discuss the
horizontal asymptotics of the meromorphic 3D-index
$$
\Imer(0,0)(q) = \sum_{n \in \BZ} \Irot(n,n)(q)
\sim \sum_{n \in \BZ}
\widehat{\Phi}^{(\s_{1})}_{n}(2\pi i\tau) \widehat{\Phi}^{(\s_{2})}_{n}(-2\pi i\tau)
\,.
$$
Equation~\eqref{41Phi7terms} implies that 
$\widehat{\Phi}^{(\s_1)}_{n}(\hbar)\widehat{\Phi}^{(\s_2)}_{n}(-\hbar)$
up to $O(\hbar^7)$ is given by
\be
\label{41Phi^2}
\begin{tiny}
\begin{aligned}
&
3^{-1/2}e^{\frac{2i\mathrm{Vol}}{\hbar}}
\exp\Big(2\sqrt{-3}\hbar\Big)^{n^2/2}\Big(1
- \frac{11}{12\sqrt{-3}^3}\hbar\\
&\quad + \Big(\frac{-9}{\sqrt{-3}^6}n^2 + \frac{409}{288\sqrt{-3}^6}\Big)\hbar^2\\
&\quad + \Big(\frac{-81}{2\sqrt{-3}^9}n^4 + \frac{159}{4\sqrt{-3}^9}n^2
- \frac{209839}{51840\sqrt{-3}^9}\Big)\hbar^3\\
&\quad + \Big(\frac{2403}{8\sqrt{-3}^{12}}n^4 - \frac{5653}{32\sqrt{-3}^{12}}n^2
+ \frac{39693941}{2488320\sqrt{-3}^{12}}\Big)\hbar^4\\
&\quad + \Big(\frac{12879}{20\sqrt{-3}^{15}}n^6 - \frac{129933}{64\sqrt{-3}^{15}}n^4
+ \frac{5481733}{5760\sqrt{-3}^{15}}n^2
- \frac{17049329117}{209018880\sqrt{-3}^{15}}\Big)\hbar^5\\
&\quad + \Big(\frac{6561}{8\sqrt{-3}^{18}}n^8 - \frac{646461}{80\sqrt{-3}^{18}}n^6
+ \frac{18984649}{1280\sqrt{-3}^{18}}n^4 - \frac{1718762189}{276480\sqrt{-3}^{18}}n^2
+ \frac{39021801798779}{75246796800\sqrt{-3}^{18}}\Big)\hbar^6 + O(\hbar^7) \Big)\,.\\
\end{aligned}
\end{tiny}
\ee
As explained in Section~\ref{sub.asymer}, the coefficent of $\hbar^k$ in
$A(\hbar):=\sum_{n\in\BZ}\widehat{\Phi}^{(\s_1)}_{n}(2\pi i\hbar)
\widehat{\Phi}^{(\s_2)}_{n}(-2\pi i\hbar)$ depends on the coefficents of the summand 
up to order $\hbar^{3k}$, and is computed by formal Gaussian integration. Then as 
$\langle 1 \rangle_{2\sqrt{-3}\hbar}=\z_8\sqrt{\frac{2\pi i}{2\sqrt{-3}\hbar}}$ 
we have
\be
A(\hbar)=e^{\frac{2i\mathrm{Vol}}{\hbar}}
\sqrt{\frac{2\pi i}{\hbar}}3^{-\frac{3}{4}}2^{-\frac{1}{2}}B(\hbar)\,,
\ee
where $B(\hbar)=1+O(\hbar)$. Explicitly, the terms
of~\eqref{41Phi^2} that contribute to the coefficient of $\hbar$ in $B(\hbar)$
are 
\be
-\frac{11}{12\sqrt{-3}^3}\hbar
- \frac{9}{\sqrt{-3}^6}n^2\hbar^2
-\frac{81}{2\sqrt{-3}^9}n^4\hbar^3 \,.
\ee
It follows that 
\be
\label{41hc1}
\begin{aligned}
\mathrm{coeff}(B(\hbar),\hbar) &=
-\frac{11}{12\sqrt{-3}^3}\hbar
- \frac{9}{\sqrt{-3}^6}
\frac{\langle n^2\rangle_{2\sqrt{-3}\hbar}}{\langle 1 \rangle_{2\sqrt{-3}\hbar}}\hbar^2
-\frac{81}{2\sqrt{-3}^9}
\frac{\langle n^4\rangle_{2\sqrt{-3}\hbar}}{\langle 1 \rangle_{2\sqrt{-3}\hbar}}
\hbar^3
\\
&=
-\frac{11}{12\sqrt{-3}^3}\hbar
- \frac{9}{\sqrt{-3}^6}(-1)\Big(\frac{1}{2\sqrt{-3}\hbar}\Big)\hbar^2
-\frac{81}{2\sqrt{-3}^9}3\Big(\frac{1}{2\sqrt{-3}\hbar}\Big)^2\hbar^3\\
&=
\frac{-19}{24\sqrt{-3}^3}\hbar\,.
\end{aligned}
\ee
Similarly, we obtain that 
\be
\label{41hc2}
\begin{tiny}
\begin{aligned}
\mathrm{coeff}(B(\hbar),\hbar^2) &=
\frac{409}{288\sqrt{-3}^6}\hbar^2
+ \frac{159}{4\sqrt{-3}^9}\frac{\langle n^2\rangle}{\langle 1 \rangle}
\hbar^3
+ \frac{2403}{8\sqrt{-3}^12}\frac{\langle n^4\rangle}{\langle 1 \rangle}
\hbar^4
+\frac{12879}{20\sqrt{-3}^{15}}\frac{\langle n^6\rangle}{\langle 1 \rangle}
\hbar^5 
+\frac{6561}{8\sqrt{-3}^{18}}\frac{\langle n^8\rangle}{\langle 1 \rangle}
\hbar^6
\\
&=\frac{409}{288\sqrt{-3}^6}\hbar^2 + \frac{159}{4\sqrt{-3}^9}(-1)
\Big(\frac{1}{2\sqrt{-3}\hbar}\Big)\hbar^3
+ \frac{2403}{8\sqrt{-3}^12}3\Big(\frac{1}{2\sqrt{-3}\hbar}\Big)^2\hbar^4
\\
&\qquad+\frac{12879}{20\sqrt{-3}^{15}}(-15)\Big(\frac{1}{2\sqrt{-3}\hbar}\Big)^3
\hbar^5+\frac{6561}{8\sqrt{-3}^{18}}105
\Big(\frac{1}{2\sqrt{-3}\hbar}\Big)^4\hbar^6\\
&=\frac{1333}{1152\sqrt{-3}^6}\hbar^2\,.
\end{aligned}
\end{tiny}
\ee
Equations~\eqref{41hc1} and~\eqref{41hc2} are in complete agreement with
the numerical extrapolation~\eqref{Imernum}. 
As we will see in Section~\ref{sec.TV}, the same rational numbers with numerator
$19$ and $1333$ will appear in the asymptotics of the Turaev--Viro invariant
of the $4_1$ knot.

\subsection{Vertical asymptotics}
\label{sub.41vertical}

In this section, we discuss the vertical asymptotics of the colored holomorphic
blocks and of the rotated 3D-index. As was mentioned in the introduction, we have
no theoretical or computational control in the $\tq$-corrections, nor to exponentially
small terms, which we will therefore ignore. 

Recall that the two colored holomorphic blocks of $4_1$ coincide, when $n=0$ with
the pair of $q$-series of the $4_1$ studied in~\cite{GZ:qseries}
(see Equation~\eqref{41gG}). In~\cite{GZ:qseries} it was found that their
vertical asymptotics $\tau \downarrow 0$ were given by
\be
\begin{tiny}
\begin{aligned}  
h^{(0)}_0(q) & \sim \; \z_8\sqrt{\tau} \left(
\wh\Phi^{(\s_1)}_0(2\pi i\tau) - \wh\Phi^{(\s_2)}_0(2\pi i\tau) \right)\,, \qquad
& 
h^{(1)}_0(q) & \sim \; \frac{\z_8}{\sqrt{\tau}} \left(
\wh\Phi^{(\s_1)}_0(2\pi i\tau) + \wh\Phi^{(\s_2)}_0(2\pi i\tau) \right) \,.
\end{aligned}
\end{tiny}
\ee
The factor $\z_8$ in the above formula
is due to a different normalisation of the constant term of the asymptotic series
in~\cite{GZ:qseries}, a famous eighth root of unity ambiguity arising from the
half integral weight. Note also that the asymptotic series in the above equation
involve an oscillatory term, namely $e^{\pm\mathrm{Vol}/(2\pi \tau)}$, times a
power series in $\tau$. We can complement the above with the asymptotics
$\tau \downarrow 0$ 
\be
\begin{tiny}
\begin{aligned}
h^{(0)}_0(q^{-1}) & \sim  \z_8\sqrt{-\tau} \left(
  \wh\Phi^{(\s_1)}_0(-2\pi i\tau) + \wh\Phi^{(\s_2)}_0(-2\pi i\tau) \right)\,,
\qquad &
h^{(1)}_0(q^{-1}) & \sim  \frac{\z_8}{\sqrt{-\tau}} \left(
-\wh\Phi^{(\s_1)}_0(-2\pi i\tau) + \wh\Phi^{(\s_2)}_0(-2\pi i\tau) \right) \,.
\end{aligned}
\end{tiny}
\ee
It comes with no surprise that for all integers $n$, we have
\be
\label{41hvasy}
\begin{tiny}
\begin{aligned}
h^{(0)}_n(q) & \sim \; \z_8\sqrt{\tau} \left(
\wh\Phi^{(\s_1)}_n(2\pi i\tau) - \wh\Phi^{(\s_2)}_n(2\pi i\tau) \right)\,,
\quad
&h^{(0)}_n(q^{-1}) & \sim \;  \z_8\sqrt{-\tau} \left(
\wh\Phi^{(\s_1)}_n(-2\pi i\tau) + \wh\Phi^{(\s_2)}_n(-2\pi i\tau) \right)\,, \\
h^{(1)}_n(q) & \sim \; \frac{\z_8}{\sqrt{\tau}} \left(
\wh\Phi^{(\s_1)}_n(2\pi i\tau) + \wh\Phi^{(\s_2)}_n(2\pi i\tau) \right)\,,
\quad
&h^{(1)}_n(q^{-1}) & \sim \; \frac{\z_8}{\sqrt{-\tau}} \left(
-\wh\Phi^{(\s_1)}_n(-2\pi i\tau) + \wh\Phi^{(\s_2)}_n(-2\pi i\tau) \right) \\
\end{aligned}
\end{tiny}
\ee
as $\tau \downarrow 0$. Note that these vertical asymptotics are
bounded and oscillatory. This, together with Equations~\eqref{41.irot}
and~\eqref{41Phi12}, implies that for all integers $n$ and $n'$
we have
\be
\begin{aligned}
\Irot(n,n')(q) & \sim
-i(\wh\Phi^{(\s_1)}_n  (2\pi i \tau) \wh\Phi^{(\s_1)}_{n'}(2\pi i \tau)
- \wh\Phi^{(\s_2)}_{n}  (2\pi i \tau) \wh\Phi^{(\s_2)}_{n'}(2\pi i \tau))
\\
& =  \wh\Phi^{(\s_1)}_n  (2\pi i \tau) \wh\Phi^{(\s_2)}_{n'}(-2\pi i \tau)
+ \wh\Phi^{(\s_2)}_{n}  (2\pi i \tau) \wh\Phi^{(\s_1)}_{n'}(-2\pi i \tau)\,,
\end{aligned}
\ee
confirming Conjecture~\ref{conj.2} for the $4_1$ knot.
We now discuss the vertical asymptotics of the meromorphic 3D-index. There
are two key problems: the first is to make sense of the asymptotic series
$\sum_{n \in \BZ} \wh\Phi^{(\s_1)}_n  (\hbar) \wh\Phi^{(\s_1)}_{n}(-\hbar)
$
which has no exponential term and it is simply a power series in $n \hbar$ and $\hbar$.
The second is that the sum of the above series for $(\s_1,\s_1)$ and for
$(\s_2,\s_2)$ simply vanishes, if we interchange the summation over $n$ with
the terms of the above series. On the other hand, the numerical vertical asymptotics of
the meromorphic 3D-index suggest a non-vanishing contribution. We now explain a
solution to this apparent contradiction. The key idea is to use a version of the
asymptotic series whose coefficients of the powers of $\hbar$ are not polynomials in
$n \hbar$ but rather branches of algebraic functions on an affine curve
(the $A$-polynomial curve). These algebraic functions were discussed in detail and
computed in the paper~\cite{GGM:peacock} whose notation we follow here.

Recall the asymptotic series $\Phi_{4_1}^{(\s)}(x,\hbar)$
from~\cite[Sec.5.1]{GGM:peacock}: it comes from a power series in $\hbar$ of the form
$\sqrt{\delta_{4_1}(x,y)} \Phi_{4_1}(x,y,\hbar) \in \BQ(X_{4_1})[[\hbar]]$ with first
few terms given by~\cite[Eqn.(123)]{GGM:peacock}
\be
\label{Phi41xyh}
\begin{aligned}
\Phi_{4_1}(x,y,\hbar) =& \frac{1}{\sqrt{\delta_{4_1}}} \bigl(1
- \frac{1}{24\,\delta_{4_1}^3}
\big(x^{-3} - x^{-2} - 2 x^{-1} + 15 -2 x - x^2 + x^3\big) \hbar
\bigr. \\ 
& \phantom{==}
+ \frac{1}{1152\,\delta_{4_1}^6} \big(x^{-6} - 2 x^{-5} - 3 x^{-4} + 610 x^{-3} - 
606 x^{-2} - 1210 x^{-1}  \big. \\
& \bigl. \big. \phantom{===}
  + 3117 - 1210 x - 606 x^2 + 610 x^3 - 3 x^4 - 2 x^5 + x^6\big) \hbar^2 + \dots
  \bigr) \,.
\end{aligned}
\ee
Here
$X_{4_1}$ is the affine curve given by the equation~\cite[Eqn.(116)]{GGM:peacock}
\be
\label{X41}
-x^2 y = (1-y)(1-xy) \,,
\ee
$\BQ(X_{4_1})=\BQ(x)[y]/(x^2 y +(1-y)(1-xy))$ is the field of rational functions
of $X_{4_1}$, and~\cite[Eqn.(118)]{GGM:peacock}
\be
\label{delta41}
\delta_{4_1}(x,y) = (xy^2-1)/(xy) \,.
\ee
Let $y_j(x)$ for $j=1,2$ denote the roots of Equation~\eqref{X41}; they are algebraic
functions that collide when $x$ is a root of the discriminant $(1-3x+x^2)(1+x+x^2)$
of~\eqref{X41} with respect to $y$. The discriminant has two real roots at
$a_0=(3-\sqrt{5})/2$ and $a_1=(3+\sqrt{5})/2$ satisfying $a_0 a_1 =1$. 
The two branches are real when $x>a_1$, and they are chosen so that
\be
y_1(a_1+1)= 0.16866\dots, \qquad
y_2(a_1+1)= -0.16866\dots \,.
\ee
Given that~\eqref{X41} is a quadratic equation in $y$, an easy calculation shows that
\be
\label{41y}
\delta_{4_1}(x,y_j(x)) = (-1)^{j-1} \sqrt{x^{-2}-2x^{-1}-1-2x+x^2}
\ee
for $j=1,2$. We now define $\Phi^{(\s_j)}(x,\hbar) = \Phi(x,y_j(x),\hbar)$ for $j=1,2$. 
In the vertical asymptotics of the meromorphic 3D-index, we use the bilinear expression
\be
\label{41PP}
\wh\Phi_{4_1}(x,y,\hbar) \wh\Phi_{4_1}(x,y,-\hbar) = \frac{1}{\delta_{4_1}} \big(
1+ \frac{1}{\delta_{4_1}^6}(x^{-3} - x^{-2} - 2x^{-1} + 5 - 2x - x^2 + x^3)
\hbar^2 + \dots \big) \,,
\ee
evaluated at $y=y_j(x)$ for $j=1,2$. The branches $y_j(x)$ satisfy the symmetry
$y_j(x)=y_j(1/x)$ for $j=1,2$ and  the linear relation
\be
1/\delta_{4_1}(x,y_1(x)) + 1/\delta_{4_1}(x,y_2(x)) = 0 \,. 
\ee
The above was called a quadratic relation of asymptotic series in~\cite{GZ:kashaev}.

The functions $y_j(x)$ are analytic functions on $\BR\setminus\{a_0,a_1\}$
and satisfy the bounds $y_j(x)=O(1/x)$ for $|x| \gg 0$ and $y_j(x)=(x-a_0)^{-1/2}$
for $x$ near $a_0$ and $y_j(x)=(x-a_1)^{-1/2}$ for $x$ near $a_1$. It follows that
the period integrals $\int_{C_k} f_j(u) du$ of $f_j(u)=1/\delta_j(e^u)$ for $j=1,2$
and for $C_1=(-\infty,\log(a_0))$, $C_2=(\log(a_0),\log(a_1))$ and
$C_3=(\log(a_1), \infty)$ are well-defined (i.e., absolutely-convergent) and in fact
they are integer linear combinations of
\be
\begin{aligned}
\int_{C_3} f_1(e^u) du &= 0.70030152116630101159\dots\,,\\
\int_{C_2} f_1(e^u) du &= -i1.5962422221317835101\dots\,.
\end{aligned}
\ee


\section{The Turaev--Viro invariant of a knot}
\label{sec.TV}


\subsection{Definition}
\label{sub.TV}

The Turaev-Viro invariant of a closed 3-manifold was defined in
~\cite{TV} using a triangulation of the manifold where its extension to triangulated
3-manifolds with nonempty boundary (in particular to link complements), 
was also discussed. It was noticed by Walker, Turaev and
Roberts~\cite{Roberts} independently that the Turaev-Viro invariant of a closed
3-manifold $M$ equals to the $\SU(2)$ Reshetikhin--Turaev invariant of its double
$D(M) = M \sharp -M$. This relation was extended to 3-manifolds with boundary
(where the double $D(M)=M \cup_{\partial M} -M$) by Benedetti and Petronio~\cite{BP}.
A careful extension of the above relation from $\SU(2)$ to $\SO(3)$ was given
by Chen--Yang in~\cite{CY} and by Detcherry--Kalfagianni--Yang~\cite[Thm.3.1]{DKY},
and a Volume Conjecture for them was formulated in~\cite{CY}. A comparison between
the Turaev--Viro invariants of a knot complement with the Turaev--Viro invariant of
the closed 3-manifold obtained by 0-surgery on the knot was given by
Detcherry--Kalfagianni in~\cite[Cor.5.3]{DK}. The comparison takes the form of an
inequality whose proof follows easily from the Cauchy-Schwarz inequality
of an inner product. 

In the present paper, we will focus on the Turaev--Viro invariant of a knot, which
is defined by a state-sum formula in terms of colorings of edges of an ideal
triangulation of the knot complement. Theorem 3.1 of~\cite{DKY} expresses this
state-sum in terms of the values of the colored Jones polynomial of the knot
at roots of unity as follows:
\be
\label{eq.TV}
\TV_{K,m+1/2}= 
\ev_{1/(m+1/2)} \left(
\eta_{2m+1}^2 \sum_{k=1}^m [k]^2 |J_{K,k}(q)|^2 \right)\,.
\ee
Here, $J_{K,n}(q) \in \BZ[q^{\pm 1}]$ denotes the $n$-th colored Jones polynomial
of a knot $K$ in 3-space~\cite{KR,Tu:YB}, normalised by $J_{\text{unknot},n}(q)=1$,
and such that $J_{K,2}(q)$ is the Jones polynomial~\cite{Jones} of $K$, and 
$\eta_{2m+1}=\frac{2}{\sqrt{2m+1}}\sin\left(\frac{2\pi}{2m+1}\right)$,
$\e(x)=e^{2 \pi i x}$, $[k]=(q^{k/2}-q^{-k/2})/(q^{1/2}-q^{-1/2})$ and
$\ev_{a/c}(f(q))=f(\e(a/c))$ for $a/c \in \BQ$.

We stress that in the definition of the Turaev--Viro invariants of a knot,
we choose as our variable $m+1/2$ rather than $2m+1$ for reasons that will become
clear later. 

It is obvious from~\eqref{eq.TV} that the Turaev--Viro invariant of a knot at $m+1/2$
is a finite sum of positive $m$ terms. Experimentally it appears that the dominant term
is near the end (i.e., $J_{K,m}(\e(1/(m+1/2))$) and grows exponentially with respect
to $m$ and has a full asymptotic expansion into power series in $1/(m+1/2)$.
Moreover, each term for $k=m-n$ with $n>0$ fixed and $n$ large also grows
exponentially with respect to $m$, at the same rate as the one with $n=0$, and
with subleading corrections which are power series in $1/(m+1/2)$ with coefficients
polynomials in $n$. Moreover, a symmetry of the colored Jones polynomial allows one
to double the sum over $m$ terms to one over $2m$ terms, and thus consider $n$
to be a fixed integer. One can replace the sum over $n$ by a formal Gaussian integral
and one can further experimentally observe the expansion of the asymptotics of
the Turaev--Viro invariant of a knot to all orders in $m+1/2$. The resulting series
have interesting arithmetic and analytic properties analogous to those
in~\cite{GZ:kashaev} with one key difference: they are the double of the series
of~\cite{GZ:kashaev}. 

\subsection{Numerical asymptotics}
\label{sub.41tvasy}


As an illustration of the above discussion, we will give the asymptotics of the
Turaev--Viro invariant of the simplest hyperbolic ($4_1$) knot.
Let $v=\Vol(4_1)/(2\pi)=2 \, \Im \Li_2(\e(1/6))/(2 \pi)=0.3230659472\dots$
denote the volume of the $4_1$ knot~\cite{Thurston}, divided by $2\pi$.

\begin{observation}
\label{prop.TV}
We have: 
\be
\label{TV21asy}
\TV_{4_1,m+1/2} \sim 
\frac{e^{2v(m+1/2)}}{2^{\frac{1}{2}} \cdot 3^{\frac{3}{4}}} (m+1/2)^{\frac{1}{2}}
B\left(\frac{\pi}{2^2 \cdot 3^2 \cdot \sqrt{3} \cdot (m+1/2)}\right) \,,
\ee
where $B(h)=\sum_{k=0}^\infty \frac{b_k}{k!} h^k$ with $b_k \in \BQ$ and
the first 10 coefficients are given by
\begin{center}
\def\arraystretch{1.5}
\begin{tabular}{|c|c|c|c|c|c|c|c|c|}\hline
$k$ & $0$ & $1$ & $2$ & $3$ & $4$ & $5$ & $6$   
\\ \hline
$b_k$ 
& $1$
& $19$
& $1333$
& $\frac{1601717}{5}$
& $\frac{704696117}{5}$
& $\frac{683123156521}{7}$
& $\frac{3461441912579591}{5 \cdot 7}$
\\ \hline
\end{tabular}
\end{center}
and
\begin{center}
\def\arraystretch{1.5}
\begin{tabular}{|c|c|c|c|}\hline
$k$ & $7$ & $8$ & $9$ 
\\ \hline
$b_k$ 
& $\frac{691282346978984873}{5}$
& $\frac{1274507463563873288357}{5}$
& $\frac{164770744067453335344413873}{5^2 \cdot 11}$
\\ \hline
\end{tabular}
\end{center}
\end{observation}

This is a numerical observation, obtained by numerical computation of the values of
the colored Jones $N$-th colored Jones polynomials at roots of unity (using either
the Le-Habiro formula for the colored Jones polynomial, or its recursion, known
to exist for all knots, and which implies that computing the $N$-th colored Jones
polynomial at a root of unity can be done so in $O(N)$-steps) for $N=1000,\dots,1100$,
to high precision, then numerically extrapolating the answer and finally recognizing
the corresponding elements of the trace field (which in our case is $\BQ(\sqrt{-3})$
hence the numbers in question are rational). A detailed description of this method
with numerous examples of asymptotics of quantum invariants is given 
in~\cite{GZ:kashaev} and in~\cite{GZ:qseries}.

Note that the growth rate of the Turaev--Viro invariant of the $4_1$ knot is not
the volume, but twice the volume. This doubling has to do with the fact that the
Turaev--Viro invariant of a knot complement is the Reshetikhin--Turaev invariant
of the doubled manifold, and matches the fact that the rotated 3D-index is a bilinear
function of the colored holomorphic blocks (see Equation~\eqref{ImerH}), and hence
its asymptotics is a bilinear (and not linear) combination of asymptotic series of
colored holomorphic blocks. 

\subsection*{Acknowledgements} 
This paper would not be possible without the generous sharing of ideas and
enlightening conversations from our friends and collaborators. Tudor Dimofte told us
about the rotated 3D-index many years ago, Jie Gu did the initial computations of
the rotated 3D-index, and Don Zagier taught us over the years the
numerical asymptotic methods and recognition of numbers which were crucial to
the numerical discoveries of our paper. In a sense, they are all co-authors of this
work. We wish to thank them all. 

In addition, we wish to thank Reneaud Detcherry and Andrew Kricker for their
encouragement. The work of C.W. has been supported by the Max-Planck-Gesellschaft.
S.G. wishes to thank the Max-Planck-Institute in Bonn and the Nanyang Technological
University in Singapore for their hospitality where the paper was completed.


\appendix

\section{The 3D-index of the $5_2$ knot}
\label{sec.52}

In this appendix, we discuss in detail the 3D-index of the $5_2$ knot and its
holomorphic blocks and colored holomorphic blocks. All functions and number in this
section depend on the knot $5_2$ which we omit from the notation.

\subsection{Holomorphic blocks}

The $5_2$ knot has three nontrivial boundary parabolic $\PSL_2(\BC)$ representations,
all Galois conjugate to the geometric one, the latter having a cubic trace field
of discriminant -23. Hence, this knot has three holomorphic blocks
$B^{(\a)}(x;q)$ for $\a=1,2,3$ (see~\cite[Eqn.6.48]{holo-blocks}
and~\cite[Sec.6.2]{GGM:peacock}) 
\be
\label{52Bx}
\begin{aligned}
B^{(1)}(x;q)
&=
\theta(x;q)G(x,x^{-1},1,q)\,,\\
B^{(2)}(x;q)
&=
\frac{\theta(x;q)}{\theta(-q^{\frac{1}{2}}x;q)}G(x,x^{2},x,q)\,,\\
B^{(3)}(x;q)
&=
\frac{\theta(x^{-1};q)}{\theta(-q^{\frac{1}{2}}x^{-1};q)}G(x^{-1},x^{-2},x^{-1},q) \,,
\end{aligned}
\ee
expressed in terms of the function
\be
G(x,y,z,q) = \Bigg\{
\begin{array}{cl}
(qx;q)_{\infty}(qy;q)_{\infty}\sum_{k=0}^{\infty}\frac{z^{k}}{(q^{-1};q^{-1})_{k}
(qx;q)_{k}(qy;q)_{k}} & \text{for }|q|<1\\
\frac{1}{(x;q^{-1})_{\infty}(y;q^{-1})_{\infty}}\sum_{k=0}^{\infty}\frac{z^{k}}
{(q^{-1};q^{-1})_{k}(qx;q)_{k}(qy;q)_{k}} & \text{for }|q|>1 \,.
\end{array}
\ee
As in the case of the $4_1$ knot, Wilf-Zeilberger theory~\cite{WZ} implemented by
Koutschan~\cite{Koutschan} implies that the holomorphic blocks satisfy
the same $q$-difference equation
\be
\label{52.Ahat}
\begin{tiny}
\begin{aligned}
P_{0}(x,q)B^{(\alpha)}(x,q)+P_{1}(x,q)B^{(\alpha)}(q^{-1}x,q)
+P_{2}(x,q)B^{(\alpha)}(q^{-2}x,q)
+P_{3}(x,q)B^{(\alpha)}(q^{-3}x,q)&=0\,,
\end{aligned}
\end{tiny}
\ee
for $\alpha=1,2,3$, where
\be
\begin{tiny}
\begin{aligned}
P_{0}(x,q)&=-q^{-2}x^2(1-q^{-2}x)(1+q^{-2}x)(1-q^{-5}x^2)\,,\\
P_{1}(x,q)&=q^{3/2}x^{-3}(1-q^{-1}x)(1+q^{-1}x)(1-q^{-5}x^2) \\
& \quad \cdot
(1-q^{-1}x-q^{-1}x^2-q^{-4}x^2+q^{-2}x^2+q^{-3}x^2+q^{-2}x^3
+q^{-5}x^3+q^{-5}x^4+q^{-5}x^4-q^{-6}x^5)\,,\\
P_{2}(x,q)&=q^{5}x^{-5}(1-q^{-2}x)(1+q^{-2}x)(1-q^{-1}x^2) \\
& \quad \cdot
(1-q^{-2}x-q^{-2}x-q^{-2}x^2-q^{-5}x^2+q^{-4}x^3+q^{-7}x^3-q^{-5}x^3
-q^{-6}x^3+q^{-7}x^4-q^{-9}x^5)\,,\\
P_{3}(x,q)&=q^{\frac{11}{2}}x^{-5}(1-q^{-1}x)(1+q^{-1}x)(1-q^{-1}x^2)\,.
\end{aligned}
\end{tiny}
\ee
Compare with~\cite[Eqn.(259)]{GGM:peacock} (with $m=0$), the slight difference
in the coefficients being due to the fact that we insist our $q$-difference equation
to be satisfied by the rotated 3D-index, as we already discussed for the case of
the $4_1$ knot.

\subsection{Colored holomorphic blocks}


The holomorphic blocks are meromorphic functions of $(x,q)$. Now we define the
colored holomorphic blocks from the expansions of the expansions of the holomorphic
blocks around $x=q^n$, or equivalently when $x=q^ne^u$, around $u=0$. 
To begin with, we define $h^{(0)}_{n}(q)$ for $|q|<1$ by
\be
B^{(1)}(q^{n}e^{u};q)
=(-q^{\frac{1}{2}};q)_{\infty}^2(q;q)_{\infty}^2h_{n}^{(0)}(q)+O(u^{1})\,.
\ee
One can then show that
\be
\begin{aligned}
B^{(2)}(q^{n}e^{u};q)
&=(-q^{\frac{1}{2}};q)_{\infty}^2h_{n}^{(0)}(q)u^{-1}+O(u^{0})\,,\\
B^{(3)}(q^{n}e^{u};q)
&=-(-q^{\frac{1}{2}};q)_{\infty}^2h_{n}^{(0)}(q)u^{-1}+O(u^{0})\,.
\end{aligned}
\ee
Next, we define $h^{(1)}_{n}(q)$ for $|q|<1$ by
\be
B^{(2)}(q^ne^u;q)+B^{(3)}(q^ne^u;q)
=-2(-q^{\frac{1}{2}};q)_{\infty}^2h_{n}^{(1)}(q)+O(u^1)\,.
\ee
Then, one can show that
\be
-B^{(1)}(q^ne^u;q)\frac{\wp'(e^u,q)}{2\wp(e^u,q)(q;q)_{\infty}^2}-B^{(2)}(q^ne^u;q)
=(-q^{\frac{1}{2}};q)_{\infty}^2h_{n}^{(1)}(q)+O(u^1) \,,
\ee
where $\wp$ is the Weierstrass $\wp$-function. 
Last, we define $h^{(2)}_{n}(q)$ for $|q|<1$ by
\be
-B^{(1)}(q^ne^u;q)\frac{\wp'(e^u,q)}{\wp(e^u,q)(q;q)_{\infty}^2}
-B^{(2)}(q^ne^u;q)+B^{(3)}(q^ne^u;q)
= (-q^{\frac{1}{2}};q)_{\infty}^2h_{n}^{(2)}(q)u+O(u^2)\,.
\ee

Next we define the colored holomorphic blocks when $|q|>1$. 
We begin by defining $h^{(0)}_{n}(q^{-1})$ for $|q|<1$ by
\be
-B^{(1)}(q^{-n}e^{-u};q^{-1})=
\frac{1}{(-q^{\frac{1}{2}};q)_{\infty}^2(q;q)_{\infty}^2}
h_{n}^{(0)}(q)u^{-2}+O(u^{-1})\,.
\ee
One can show that
\be
\begin{aligned}
-B^{(2)}(q^{-n}e^{-u};q^{-1})
&=\frac{1}{(-q^{\frac{1}{2}};q)_{\infty}^2}h_{n}^{(0)}(q)u^{-1}+O(u^{0})\,,\\
-B^{(3)}(q^{-n}e^{-u};q^{-1})
&=-\frac{1}{(-q^{\frac{1}{2}};q)_{\infty}^2}h_{n}^{(0)}(q)u^{-1}+O(u^{0})\,.
\end{aligned}
\ee
Then, we define $h^{(1)}_{n}(q^{-1})$ for $|q|<1$ by
\be
B^{(2)}(q^{-n}e^{-u};q)+B^{(3)}(q^{-n}e^{-u};q)
=
\frac{1}{(-q^{\frac{1}{2}};q)_{\infty}^2}h^{(1)}_{n}(q)+O(u^1)\,.
\ee
One can show
\be
-2B^{(1)}(q^{-n}e^{-u};q)
\frac{\wp(e^u;q)(q;q)_{\infty}^2}{\wp'(e^u;q)}-2B^{(2)}(q^{-n}e^{-u};q)
=
 \frac{1}{(-q^{\frac{1}{2}};q)_{\infty}^2}h^{(1)}_{n}(q)+O(u^1)\,.
\ee
Finally, we define $h^{(2)}_{n}(q^{-1})$ for $|q|<1$ by
\be
\begin{small}
\begin{aligned}
2B^{(1)}(q^{-n}e^{-u};q)\frac{\wp(e^u;q)(q;q)_{\infty}^2}{\wp'(e^u;q)}+B^{(2)}
(q^{-n}e^{-u};q)-B^{(3)}(q^{-n}e^{-u};q)
=
\frac{h^{(2)}_{n}(q)}{(-q^{\frac{1}{2}};q)_{\infty}^2}u+O(u^2)\,.
\end{aligned}
\end{small}
\ee
Since the holomorphic blocks $B^{(\a)}(x;q)$ satisfy the $q$-difference
equation~\eqref{52.Ahat}, it follows that the colored holomorphic blocks, given
in~\eqref{52.h0},~\eqref{52.h1}, and~\eqref{52.h2}, satisfy the linear
$q$-difference equation
\be
\label{52.Ahat.colored}
\begin{aligned}
P_{0}(q^n,q)h^{(\alpha)}_{n}(q)+P_{1}(q^n,q)h^{(\alpha)}_{n-1}(q)
+P_{2}(q^n,q)h^{(\alpha)}_{n-2}(q)
+P_{3}(q^n,q)h^{(\alpha)}_{n-3}(q)&=0\,,
\end{aligned}
\ee
for all $\a=1,2,4$ and all integers $n$.

Our next lemma gives explicit $q$-hypergeometric sums for the colored holomorphic
blocks of the $5_2$ knot. Although the derivation of these sums is rigourous, albeit
routine, the formulas are lengthy. To avoid clutter, we will introduce some notation
for $q$-harmonic sums. The standard notation $H_n=\sum_{j=1}^n 1/j$ for harmonic
numbers,
and its extension $H_n^{(k)}=\sum_{j=1}^n 1/j^k$ leads to several
recent $q$-generalisations whose notation is not standard, see e.g.,
Singer~\cite{Singer:qMZV}. In our formulas for the colored holomorphic blocks
of the $5_2$ knot, we will use
\be
H_n(q) = \sum_{j=1}^n \frac{q^j}{1-q^j}, \qquad
H^{(2)}_n(q) = \sum_{j=1}^n \frac{q^j}{(1-q^j)^2} \,.
\ee
\begin{lemma}
\label{lem.52h}
We have:
\be
\label{52.h0}
\begin{aligned}
h_{n}^{(0)}(q)
&=
(-1)^nq^{|n|/2}\sum_{k=0}^{\infty}
\frac{q^{|n|k}}{(q^{-1};q^{-1})_{k}(q;q)_{k+2|n|}(q;q)_{k+|n|}}\,,
\end{aligned}
\ee
\be\label{52.h1}
\begin{aligned}
h^{(1)}_n(q)
&=
-(-1)^nq^{|n|/2}\sum_{k=0}^{\infty}
\frac{q^{|n|k}}{(q;q)_{k+2|n|}(q^{-1};q^{-1})_{k}(q;q)_{k+|n|}}\\
&\qquad\times\left(k+|n|-\frac{1}{4}-3E_{1}(q)
  +H_{k}(q)
  +H_{k+|n|}(q)
  +H_{k+2|n|}(q)\right)\\
&\quad+q^{-n^2/2}\sum_{k=0}^{|n|-1}
\frac{(q^{-1},q^{-1})_{|n|-1-k}}{(q^{-1},q^{-1})_{k}(q;q)_{k+|n|}}\,,
\end{aligned}
\ee
and
\be\label{52.h2}
\begin{aligned}
h^{(2)}_{n}(q)
&=
(-1)^nq^{|n|/2}\sum_{k=0}^{\infty}
\frac{q^{|n|k}}{(q^{-1};q^{-1})_{k}(q;q)_{k+|n|}(q;q)_{k+2|n|}}\\
&\qquad\times\Bigg(
E_{2}(q)+\frac{1}{8}
-H_{k}^{(2)}(q)
-H_{k+|n|}^{(2)}(q)
-H_{k+2|n|}^{(2)}(q)\\
&\qquad\qquad-\bigg(k+|n|-\frac{1}{4}-3E_{1}(q)
+H_{k}(q)
+H_{k+|n|}(q)
+H_{k+2|n|}(q)\bigg)^2\Bigg)\\
&+
2q^{-n^2/2}\sum_{k=0}^{|n|-1}
\frac{(q^{-1},q^{-1})_{|n|-1-k}}{(q^{-1},q^{-1})_{k}(q;q)_{k+|n|}}\\
&\qquad\times\Bigg(|n|-\frac{3}{4}-3E_{1}(q)
+H_{k}(q)
+H_{k+|n|}(q)
+H_{|n|-k-1}(q)\Bigg)\\
&-
2(-1)^{n}q^{-|n|/2}\sum_{k=0}^{|n|-1}q^{-|n|k}
\frac{(q^{-1};q^{-1})_{2|n|-k-1}(q^{-1};q^{-1})_{|n|-k-1}}{(q^{-1};q^{-1})_{k}}\,,
\end{aligned}
\ee
for $|q| \neq 1$.
The colored holomorphic blocks satisfy the symmetries
\be
\label{52hsym2}
\hb_{-n}^{(\a)}(q) = \hb_{n}^{(\a)}(q), \qquad \a=0,1,2 \,.
\ee
\end{lemma}

The first few values of the colored holomorphic blocks are given by
\be
\begin{tiny}
\begin{aligned}
h_{0}^{(0)}(q)
&=
1 - q - 3q^2 - 5q^3 - 7q^4 - 6q^5 - 3q^6 + 8q^7 + 24q^8 + 52q^9 + \dots\,,\\
h_{1}^{(0)}(q)
&=
- q^{\frac{1}{2}} - 2q^{\frac{3}{2}} - 3q^{\frac{5}{2}} - 3q^{\frac{7}{2}} - 
q^{\frac{9}{2}} + 4q^{\frac{11}{2}} + 14q^{\frac{13}{2}} + 29q^{\frac{15}{2}}
+ 51q^{\frac{17}{2}} + \dots\,,\\
h_{2}^{(0)}(q)
&=
q + 2q^2 + 5q^3 + 8q^4 + 14q^5 + 19q^6 + 26q^7 + 29q^8 + 30q^9 + \dots\,,\\
\end{aligned}
\end{tiny}
\ee
\be
\begin{tiny}
\begin{aligned}
h_{0}^{(1)}(q)
&=-\frac{1}{2} + \frac{9}{2}q + \frac{21}{2}q^2 + \frac{19}{2}q^3 + \frac{9}{2}q^4 
- 27q^5 - \frac{133}{2}q^6 - 156q^7 - 252q^8 - 384q^9+\dots\,,\\
h_{1}^{(1)}(q)
&=q^{-\frac{1}{2}} + \frac{5}{2}q^{\frac{1}{2}} + 3q^{\frac{3}{2}}
- \frac{1}{2}q^{\frac{5}{2}}
 - \frac{23}{2}q^{\frac{7}{2}} - \frac{73}{2}q^{\frac{9}{2}} - 74q^{\frac{11}{2}} - 
 133q^{\frac{13}{2}} - \frac{393}{2}q^{\frac{15}{2}}
 - \frac{529}{2}q^{\frac{17}{2}}+\dots\,,\\
h_{2}^{(1)}(q)
&=-q^{-3} - 2q^{-1} - 2 - \frac{15}{2}q - 11q^2 - \frac{41}{2}q^3 - 23q^4 - 27q^5 - 
\frac{19}{2}q^6 + 30q^7 + \frac{237}{2}q^8 + 272q^9+\dots\,,\\
\end{aligned}
\end{tiny}
\ee
\be
\begin{tiny}
\begin{aligned}
h_{0}^{(2)}(q)
&=-\frac{1}{6} + \frac{37}{6}q + \frac{17}{2}q^2 - \frac{115}{6}q^3
- \frac{389}{6}q^4 - 
181q^5 - \frac{579}{2}q^6 - \frac{1414}{3}q^7 - 548q^8 - \frac{1418}{3}q^9 +\dots\,,\\
h_{1}^{(2)}(q)
&=-2q^{-\frac{3}{2}} + 4q^{-\frac{1}{2}} + \frac{1}{6}q^{\frac{1}{2}}
- \frac{29}{3}q^{\frac{3}{2}}
- \frac{59}{2}q^{\frac{5}{2}} - \frac{139}{2}q^{\frac{7}{2}}
- \frac{755}{6}q^{\frac{9}{2}}
- \frac{560}{3}q^{\frac{11}{2}} - \frac{673}{3}q^{\frac{13}{2}}
- \frac{941}{6}q^{\frac{15}{2}}+\dots\,,\\
h_{2}^{(2)}(q)
&=-2q^{-8} + 4q^{-7} - 2q^{-4} - 6q^{-3} + 6q^{-2} - 4q^{-1} + 4
+ \frac{11}{6}q+\dots\,,\\
\end{aligned}
\end{tiny}
\ee
\be
\begin{tiny}
\begin{aligned}
h_{0}^{(0)}(q^{-1})
&=1 + q^2 + 3q^3 + 6q^4 + 10q^5 + 16q^6 + 24q^7 + 37q^8 + 55q^9 +\dots\,,\\
h_{1}^{(0)}(q^{-1})
&=+q^{\frac{7}{2}} + 2q^{\frac{9}{2}} + 4q^{\frac{11}{2}} + 6q^{\frac{13}{2}}
+ 10q^{\frac{15}{2}}
+ 15q^{\frac{17}{2}} + 24q^{\frac{19}{2}} + 37q^{\frac{21}{2}} + 58q^{\frac{23}{2}}
+ 88q^{\frac{25}{2}}+\dots\,,\\
h_{2}^{(0)}(q^{-1})
&=q^{12} + 2q^{13} + 5q^{14} + 9q^{15} + 17q^{16} + 27q^{17} + 45q^{18} + 68q^{19}
+ 105q^{20}
 + 154q^{21}+\dots\,,\\
\end{aligned}
\end{tiny}
\ee
\be
\begin{tiny}
\begin{aligned}
h_{0}^{(1)}(q^{-1})
&=1 - 3q - 3q^2 + 3q^3 + 6q^4 + 12q^5 + 5q^6 + 3q^7 - 12q^8 - 25q^9 +\dots\,,\\
h_{1}^{(1)}(q^{-1})
&=-q^{\frac{3}{2}} - q^{\frac{5}{2}} + 2q^{\frac{7}{2}} + 4q^{\frac{9}{2}}
+ 6q^{\frac{11}{2}}
+ 3q^{\frac{13}{2}} - q^{\frac{15}{2}} - 9q^{\frac{17}{2}} - 18q^{\frac{19}{2}}
- 23q^{\frac{21}{2}}+\dots\,,\\
h_{2}^{(1)}(q^{-1})
&=q^5 + q^7 - q^8 - q^9 - 4q^{10} - 6q^{11} - 6q^{12} - 6q^{13} - 2q^{14}+\dots\\
\end{aligned}
\end{tiny}
\ee
\be
\begin{tiny}
\begin{aligned}
h_{0}^{(2)}(q^{-1})
&=-\frac{5}{6} + 5q - \frac{53}{6}q^2 - \frac{117}{2}q^3 - 117q^4 - \frac{601}{3}q^5 - 
\frac{865}{3}q^6 - 449q^7 - \frac{4523}{6}q^8 - \frac{7547}{6}q^9 +\dots\,,\\
h_{1}^{(2)}(q^{-1})
&=2q^{\frac{1}{2}} +q^{\frac{3}{2}} -q^{\frac{5}{2}} - \frac{119}{6}q^{\frac{7}{2}}
 - \frac{107}{3}q^{\frac{9}{2}} - \frac{172}{3}q^{\frac{11}{2}} - 62q^{\frac{13}{2}}
 - \frac{304}{3}q^{\frac{15}{2}} - \frac{349}{2}q^{\frac{17}{2}}
 - 370q^{\frac{19}{2}}+\dots\,,\\
h_{2}^{(2)}(q^{-1})
&=-2q + 4q^2 - 2q^3 - 2q^4 - 5q^5 + 2q^6 + 5q^7 + 15q^8 + 25q^9 +\dots\,.
\end{aligned}
\end{tiny}
\ee

The colored holomorphic blocks have $q$ and $q^{-1}$-degree 

\be
\begin{aligned}
\deg_q \hb^{(0)}_{n}(q)&=|n|/2\,,\qquad&
\deg_q \hb^{(0)}_{n}(q^{-1})&=|n|(5|n|+2)/2\,,\\
\deg_q \hb^{(1)}_{n}(q)&=-|n|(2|n|-1)/2\,,\qquad&
\deg_q \hb^{(1)}_{n}(q^{-1})&=|n|(2|n|+1)\,,\\
\deg_q \hb^{(2)}_{n}(q)&=-|n|(5|n|-2)/2\,,\qquad&
\deg_q \hb^{(2)}_{n}(q^{-1})&=|n|/2\,.
\end{aligned}
\ee

This can easily be deduced from Equations~\eqref{52.h0},~\eqref{52.h1},
and~\eqref{52.h2}.

\subsection{The rotated and the meromorphic 3D-index}

We next express the rotated 3D-index in terms of the colored holomorphic blocks.

\begin{proposition}
\label{prop.52Irot}
For all integers $n$ and $n'$ we have:
\be
\label{52Irot}  
\Irot(n,n')(q)
=
-\frac{1}{2}h_{n'}^{(0)}(q^{-1})h_{n}^{(2)}(q)-h_{n'}^{(1)}(q^{-1})h_{n}^{(1)}(q)
-\frac{1}{2}h_{n'}^{(2)}(q^{-1})h_{n}^{(0)}(q)\,.
\ee
\end{proposition}

\begin{proof}
The rotated 3D-index is given by the limit $x \to 1$ in Equation~\eqref{IrotB}.
On the other hand, the expansions of the holomorphic blocks in terms of coloured
holomorphic blocks, imply that
\be
\begin{tiny}
\begin{aligned}
&B^{(1)}(q^{-n'}e^{-u};q^{-1})B^{(1)}(q^{n}e^u;q)+B^{(2)}(q^{-n'}e^{-u};q^{-1})
B^{(2)}(q^{n}e^u;q)+B^{(3)}(q^{-n'}e^{-u};q^{-1})B^{(3)}(q^{n}e^u;q)\\
&=
-\frac{1}{2}h_{n'}^{(0)}(q^{-1})h_{n}^{(2)}(q)-h_{n'}^{(1)}(q^{-1})h_{n}^{(1)}(q)
-\frac{1}{2}h_{n'}^{(2)}(q^{-1})h_{n}^{(2)}(q)+O(u)\,.
\end{aligned}
\end{tiny}
\ee
The result follows. 
\end{proof}

Using Equations~\eqref{52.h0}, \eqref{52.h1}, \eqref{52.h2} and the above proposition,
it follows that the first few values of the rotated 3D-index 
$\Irot(n,n)(q)$
are given by
\be
\begin{small}
\begin{aligned}
\Irot(0,0)(q)&=
1 - 12q + 3q^2 + 74q^3 + 90q^4 + 33q^5 - 288q^6 - 684q^7 - 1095q^8 - 1140q^9 +\dots\,,
\\
\Irot(1,1)(q)&=
2q + 7q^2 + 7q^3 - 13q^4 - 68q^5 - 154q^6 - 220q^7 - 165q^8 + 157q^9 + 898q^{10} + \dots\,,
\\
\Irot(2,2)(q)&=
2q^2 + 6q^4 + 2q^5 + 17q^6 + 14q^7 + 36q^8 + 21q^9 + 6q^{10} - 110q^{11} + \dots\,,
\end{aligned}
\end{small}
\ee
and their degree is given by $\deg_q \Irot(n,n)(q)=|n|$. This, together with 
Equation~\eqref{ImerIrot} gives the first few terms of the meromorphic 3D-index
as stated in Equation~\eqref{52.mer}.

\subsection{Quadratic relations}

We now discuss a new phenomenon of the coloured holomorphic blocks of the $5_2$
knot, which was trivial for the case of the amphichiral knot $4_1$, namely
quadratic relations among the colored holomorphic blocks. These relations were
originally discovered in~\cite{GZ:qseries}, and interpreted in terms of a duality
statement of a $q$-holonomic module by the authors in~\cite{GW:qmod}.
These quadratic relations for $5_2$ are
\be
\label{52quad}
h_{n'}^{(0)}(q^{-1})h_{n}^{(2)}(q) - 2h_{n'}^{(1)}(q^{-1})h_{n}^{(1)}(q)+
h_{n'}^{(2)}(q^{-1})h_{n}^{(0)}(q) = 0 \qquad (\text{if} \,\,\, n=n') \,,
\ee
else the left hand-side is in $\BZ[q^{\pm 1/2}]$ if $n \neq n'$.
For instance, if $H_3(q) = \begin{tiny}\begin{pmatrix}
h_{0}^{(0)}(q) & h_{0}^{(1)}(q) & h_{0}^{(2)}(q)\\
h_{1}^{(0)}(q) & h_{1}^{(1)}(q) & h_{1}^{(2)}(q)\\
h_{2}^{(0)}(q) & h_{2}^{(1)}(q) & h_{2}^{(2)}(q)
\end{pmatrix} \end{tiny}$, then we have
\be
\label{52quad2}
\begin{tiny}
\begin{aligned}
H_3(q)
\begin{pmatrix}
0 & 0 & -1\\
0 & 2 & 0\\
-1 & 0 & 0
\end{pmatrix}
H_3(q^{-1})^t = & \\ & \hspace{-5cm} 
\begin{pmatrix}
0 & -2q^{\frac{1}{2}} + 2q^{\frac{3}{2}} & 2q - 6q^2 + 2q^4 + 2q^5 + 2q^6 - 4q^7 
 + 2q^8\\
2q^{-\frac{3}{2}} - 2q^{-\frac{1}{2}} & 
0 & -2q^{\frac{3}{2}} + 2q^{\frac{9}{2}}\\
2q^{-8} - 4q^{-7} + 2q^{-6} + 2q^{-5} +
 2q^{-4} - 6q^{-2} + 2q^{-1} & 2q^{-\frac{9}{2}} - 2q^{-\frac{3}{2}} & 0
\end{pmatrix}\,.
\end{aligned}
\end{tiny}
\ee
The above relations together with Proposition~\ref{prop.52Irot} imply that
\be
\Irot(n,n)(q)
=
-2h_{n}^{(1)}(q^{-1})h_{n}^{(1)}(q)\,.
\ee

\subsection{Asymptotic series}
\label{sub.52asy}

Applying the discrete WKB ansatz~\eqref{Ahat.ansatz} to the linear $q$-difference
equation~\eqref{52.Ahat.colored}, we find that $a_{2,-1}$ satisfies the polynomial
equation
\be
-32a_{2,-1}^3 - 112a_{2,-1}^2 - 120a_{2,-1} - 68=0\,.
\ee
It follows that $a_{2,-1}=-\tfrac{3}{2}+\xi$, where
\be
\label{52xi}
\xi^3-\xi^2+1=0
\ee
generates the trace field of $5_2$, the cubic field of discriminant -23.
Using the initial condition $\Phi_{0}(\hbar)$ from~\cite{GZ:kashaev}
(see also~\cite{GGM}), we find that for
\be
\delta=3\xi-2\,,
\ee
we have,
\be
\label{52Phinh}
\begin{tiny}
\begin{aligned}
\widehat{\Phi}^{(\s_1)}_{n}(\hbar)
&  =\;\frac{1}{\sqrt{\delta}} e^{\frac{V^{(\s_1)}}{\hbar}}
\exp\left(\left(-\frac{3}{2}+\xi\right)n^2\hbar\right) \\
& \times\Bigg(1 + \bigg(-\frac{33}{2116}\xi^2 - \frac{121}{1058}\xi
+ \frac{245}{2116}\bigg)\hbar\\
&\quad + \bigg(\Big(-\frac{7}{46}\xi^2 + \frac{7}{46}\xi - \frac{1}{23}\Big)n^2
+ \Big(\frac{10025}{389344}\xi^2 - \frac{12643}{389344}\xi
+ \frac{683}{97336}\Big)\bigg)\hbar^2\\
&\quad + \bigg(\Big(-\frac{5}{276}\xi^2 + \frac{17}{276}\xi - \frac{1}{69}\Big)n^4
+ \Big(-\frac{5557}{292008}\xi^2 + \frac{2609}{146004}\xi - \frac{1984}{36501}\Big)n^2\\
&\qquad + \Big(\frac{50198891}{12357778560}\xi^2 - \frac{3544387}{1235777856}\xi
+ \frac{6584729}{1029814880}\bigg)\hbar^3\\
&\quad + \bigg(\Big(-\frac{251}{8464}\xi^2 + \frac{67}{1587}\xi
- \frac{103}{3174}\Big)n^4 + \Big(\frac{175913}{26864736}\xi^2
- \frac{71191}{17909824}\xi + \frac{141915}{17909824}\Big)n^2\\
&\qquad + \Big(-\frac{952485893}{1136915627520}\xi^2
+ \frac{1861268771}{4547662510080}\xi + \frac{203137333}{909532502016}\Big)\bigg)\hbar^4
\end{aligned}
\end{tiny}
\ee
\be
\notag
\begin{tiny}
\begin{aligned}
&\quad + \bigg(\Big(\frac{259}{38088}\xi^2 - \frac{773}{47610}\xi
+ \frac{2209}{190440}\Big)n^6 + \Big(-\frac{4897763}{80594208}\xi^2
+ \frac{30391187}{322376832}\xi - \frac{13140721}{322376832}\Big)n^4\\
&\qquad + \Big(\frac{24639499193}{852686720640}\xi^2
- \frac{78232585819}{1705373441280}\xi + \frac{3754956391}{213171680160}\Big)n^2\\
&\qquad + \Big(-\frac{12690681719899}{4393041984737280}\xi^2
+ \frac{32954346270143}{8786083969474560}\xi
- \frac{3038166526261}{2196520992368640}\Big)\bigg)\hbar^5\\
&\quad + \bigg(\Big(\frac{1}{828}\xi^2 - \frac{7}{6624}\xi + \frac{7}{6624}\Big)n^8
+ \Big(\frac{80599}{3504096}\xi^2 - \frac{76921}{3504096}\xi
+ \frac{4732}{182505}\Big)n^6\\
&\qquad + \Big(-\frac{5845444967}{74146671360}\xi^2
+ \frac{12028059677}{148293342720}\xi - \frac{8174917009}{74146671360}\Big)n^4\\
&\qquad + \Big(\frac{208659394649}{5457195012096}\xi^2
- \frac{1077794714231}{27285975060480}\xi + \frac{1449480552719}{27285975060480}\Big)n^2\\
&\qquad + \Big(-\frac{8106425420368231}{2694399083972198400}\xi^2
+ \frac{25148935902272269}{8083197251916595200}\xi
- \frac{4471134896235023}{898133027990732800}\Big)\bigg)\hbar^6 + O(\hbar^7)\Bigg) \,.
\end{aligned}
\end{tiny}
\ee
Equation~\eqref{52Phinh} above implies that for $5_2$ and
$(\s,\s')=(\s_1,\s_2)$, we have
\be
\label{52Phinsqh}
\begin{tiny}
\begin{aligned}
&\widehat{\Phi}^{(\s_1)}_{n}(\hbar)\widehat{\Phi}^{(\s_2)}_{n}(-\hbar)\\
&=\frac{1}{\sqrt{\delta^{(\s_1)}\delta^{(\s_2)}}} e^{\frac{V^{(\s_1)}-V^{(\s_2)}}{\hbar}}
\exp\left(\left(\xi^{(\s_1)}-\xi^{(\s_2)}\right)n^2\hbar\right)
\Bigg(1\\
&\quad+\bigg(-\frac{33}{2116}(\xi^{(\s_1)})^2 - \frac{121}{1058}\xi^{(\s_1)}
+ \frac{33}{2116}(\xi^{(\s_2)})^2 + \frac{121}{1058}\xi^{(\s_2)}\bigg)\hbar\\
&\quad + \bigg(\Big(-\frac{7}{46}(\xi^{(\s_1)})^2 + \frac{7}{46}\xi^{(\s_1)}
- \frac{7}{46}(\xi^{(\s_2)})^2 + \frac{7}{46}\xi^{(\s_2)} - \frac{2}{23}\Big)n^2\\
&\qquad + \Big(-\frac{1089}{4477456}(\xi^{(\s_2)})^2 - \frac{3993}{2238728}\xi^{(\s_2)}
+ \frac{246745}{8954912}(\xi^{(\s_1)})^2 - \frac{3993}{2238728}(\xi^{(\s_2)})^2\\
&\qquad\qquad - \frac{14641}{1119364}\xi^{(\s_2)} - \frac{172209}{8954912}\xi^{(\s_1)}
+ (\frac{246745}{8954912}(\xi^{(\s_2)})^2 - \frac{172209}{8954912}\xi^{(\s_2)}
+ \frac{2811}{4477456}\Big)\bigg)\hbar^2\\
&\quad + \bigg(\Big(-\frac{5}{276}(\xi^{(\s_1)})^2 + \frac{17}{276}\xi^{(\s_1)}
+ \frac{5}{276}(\xi^{(\s_2)})^2 - \frac{17}{276}\xi^{(\s_2)}\Big)n^4\\
&\qquad + \Big(-\frac{1925}{97336}\xi^{(\s_2)} - \frac{107}{146004}(\xi^{(\s_1)})^2
+ \frac{1925}{97336}(\xi^{(\s_2)})^2 + \frac{1525}{292008}\xi^{(\s_1)}
+ (\frac{107}{146004}(\xi^{(\s_2)})^2 - \frac{1525}{292008}\xi^{(\s_2)})\Big)n^2\\
&\qquad + \Big(\frac{2843269}{823851904}\xi^{(\s_2)}
+ \frac{3001169}{3089444640}(\xi^{(\s_1)})^2 - \frac{2843269}{823851904}(\xi^{(\s_2)})^2\\
&\qquad\qquad + \frac{220399}{2471555712}\xi^{(\s_1)}
+ (-\frac{3001169}{3089444640}(\xi^{(\s_2)})^2
- \frac{220399}{2471555712}\xi^{(\s_2)})\Big)\bigg)\hbar^3+O(\hbar^4)\Bigg) \,.
\end{aligned}
\end{tiny}
\ee
The coefficients of $\hbar^k$ for $k=4,5,6$ in the above expression can be
computed from~\eqref{52Phinh}, but they are lengthy and will not be given here.

As explained in Section~\ref{sub.asymer}, the coefficent of $\hbar^k$ in
$A(\hbar):=\sum_{n\in\BZ}\widehat{\Phi}^{(\s_1)}_{n}(2\pi i\hbar)
\widehat{\Phi}^{(\s_2)}_{n}(-2\pi i\hbar)$ depends on the summand up to
$O(\hbar^{3k+1})$, and is computed by formal Gaussian integration. Explicitly,
we have
\be
\label{52hc1}
\begin{tiny}
\begin{aligned}
A(\hbar)&=e^{\frac{2i\mathrm{Vol}}{\hbar}}
\sqrt{\frac{2\pi i}{\hbar}}\frac{1}{\sqrt{\delta^{(\s_1)}\delta^{(\s_2)}(2\xi^{(\s_1)}-2\xi^{(\s_2)})}} \\
& \times \Bigg( 1 +
\Big(-\frac{1258}{13225}(\xi^{(\s_1)}-\xi^{(\s_2)})
+\frac{7}{1840}(\xi^{(\s_1)}-\xi^{(\s_2)})^3
-\frac{963}{211600}(\xi^{(\s_1)}-\xi^{(\s_2)})^5 \Big) \hbar \\
& +\Big(
-\frac{226591}{19467200}
+\frac{88839}{7786880}(\xi^{(\s_1)}-\xi^{(\s_2)})^2
+\frac{38187}{155737600}(\xi^{(\s_1)}-\xi^{(\s_2)})^4 \Big) \hbar^2 + O(\hbar^3)
\Bigg) \,,
\end{aligned}
\end{tiny}
\ee
in complete agreement with the numerical extrapolation of the asymptotics of
$\Irot_{5_2}(0,0)(q)$.


\subsection{Horizontal asymptotics}
\label{sub.52horizontal}

We now discuss the horizontal asymptotics of the colored holomorphic blocks
when $\tau$ tends to zero just above (\emph{resp.} below) the positive
(\emph{resp.} negative) reals. The next identity~\eqref{52holborel} expresses
the holomorphic blocks as Borel resummations of the three asymptotic series.
This identity was discovered and verified in~\cite[Sec.6.4]{GGM} when $n=$ using
about 200 coefficients (of powers of $\hbar$) of the asymptotic series. When $n$
is an arbitrary integer, we use about 90 coefficients of the asymptotic series,
each a polynomial of $n$ to numerically evaluate the Borel resummation using Pad\'e
approximants. We then find numerically that for $\tau$ in a cone
$C = \arg(\th) \in (\th_0,\th_1)$ with $0< \th_0 < \th_1$ sufficiently small,
we have
\be
\label{52holborel}
\begin{aligned}
h_{n}^{(j)}(q) &= \tau^{1-j}\sum_{k=1}^{3}
s \widehat{\Phi}^{(\s_{k})}_{n}(2\pi i\tau)H_{\s_k}^{(j)}(\tq)\,,
\end{aligned}
\ee
where for $|\tq|<1$, we have
\be
\begin{tiny}
\begin{aligned}
&\begin{pmatrix}
H_{\s_1}^{(j)}(\tq)\\
H_{\s_2}^{(j)}(\tq)\\
H_{\s_3}^{(j)}(\tq)
\end{pmatrix}
=
\begin{pmatrix}
\frac{2\tq^3 - 2\tq^2 - 2\tq + 1}{-\tq^3 + \tq^2 + \tq - 1}
&
\frac{-\tq^{11/2} + 2\tq^{9/2} - \tq^{7/2} + 2\tq^{5/2}
  - \tq^{3/2} + \tq^{1/2}}{-\tq^4 + \tq^3 + \tq - 1}
 &
 -\frac{\tq^7}{-\tq^5 + \tq^3 + \tq^2 - 1}\\
 -1&0&0\\
 \frac{-\tq^3 + \tq - 1}{-\tq^3 + \tq^2 + \tq - 1}
 &
 \frac{-\tq^{11/2} + 3\tq^{9/2} - \tq^{7/2} + 2\tq^{5/2}
   - 2\tq^{3/2} + \tq^{1/2}}{-\tq^4 + \tq^3 + \tq - 1}
 &
 -\frac{\tq^7}{-\tq^5 + \tq^3 + \tq^2 - 1}
  \end{pmatrix}
  \begin{pmatrix}
 h_{0}^{(j)}(\tq)\\
 h_{1}^{(j)}(\tq)\\
 h_{2}^{(j)}(\tq)
  \end{pmatrix}\,,
\end{aligned}
\end{tiny}
\ee
and
\be
\begin{tiny}
\begin{aligned}
&\begin{pmatrix}
H_{\s_1}^{(j)}(\tq^{-1})\\
H_{\s_2}^{(j)}(\tq^{-1})\\
H_{\s_3}^{(j)}(\tq^{-1})
\end{pmatrix}
=
\begin{pmatrix}
-1&0&0\\
\frac{-3\tq^3 + 2\tq^2 + 2\tq - 2}{-\tq^3 + \tq^2 + \tq - 1}
&
\frac{\tq^5 - \tq^4 + 2\tq^3 - \tq^2 + 2\tq - 1}{-\tq^{11/2}
   + \tq^{9/2} + \tq^{5/2} - \tq^{3/2}}
&
-\frac{1}{-\tq^7 + \tq^5 + \tq^4 - \tq^2}\\
 \frac{-\tq^3 + \tq^2 - 1}{-\tq^3 + \tq^2 + \tq - 1}
&
\frac{\tq^5 - 2\tq^4 + 2\tq^3 - \tq^2 + 3\tq - 1}{-\tq^{11/2}
   + \tq^{9/2} + \tq^{5/2} - \tq^{3/2}}
&
 -\frac{1}{-\tq^7 + \tq^5 + \tq^4 - \tq^2}
\end{pmatrix}
\begin{pmatrix}
h_{0}^{(j)}(\tq^{-1})\\
h_{1}^{(j)}(\tq^{-1})\\
h_{2}^{(j)}(\tq^{-1})
\end{pmatrix} \,.
\end{aligned}
\end{tiny}
\ee
With around 90 coefficients of $\widehat{\Phi}^{(\s_{k})}_{n}$, these identities can be checked to about 10 coefficients of $\tq$.
Using Equation~\eqref{52Irot} and the above, we obtain a bilinear expression for
$\Irot(n,n')(q)$ in terms of products
$\wh\Phi^{(\s)}_n(2\pi i\tau) \wh\Phi^{(\s')}_n(-2\pi i\tau)$ times $\tq$-series.
On the fixed ray of $\tau$ near the positive real axis, we can ignore all $\tq$-terms,
and among all bilinear combinations of asymptotic series, there is precisely one
with  $(\s,\s')=(\s_1,\s_2)$ which is exponentially larger than the others.
In particular, we find that
\be
\label{52.hnnasy3}
\Irot(n,n')(q)
\sim
\widehat{\Phi}^{(\s_{1})}_{n}(2\pi i\tau) \widehat{\Phi}^{(\s_{2})}_{n'}(-2\pi i\tau)\,,
\ee
concluding the one part of Conjecture~\ref{conj.1}.

\subsection{Vertical asymptotics}
\label{sub.52vertical}

Following the method of~\cite{GZ:qseries} and using a high precision computation
of the colored holomorphic blocks when $\tau$ is small on the imaginary axis, 
we find numerically that the coloured holomorphic blocks have the following
vertical asymptotics as $\tau \downarrow 0$ 
\be
\label{52hvasy}
\begin{tiny}
\begin{aligned}
h^{(0)}_n(q) & \sim \;  \tau\:
  \wh\Phi^{(\s_1)}_n(2\pi i\tau) - \tau\:\wh\Phi^{(\s_2)}_n(2\pi i\tau) \,,
&&\quad
h^{(0)}_n(q^{-1}) \hspace{-0.13in}&& \sim  -\tau\:\wh\Phi^{(\s_3)}_n(-2\pi i\tau) \,,\\
h^{(1)}_n(q) & \sim \; -\frac{1}{2}
\wh\Phi^{(\s_1)}_n(2\pi i\tau) -\frac{1}{2} \wh\Phi^{(\s_2)}_n(2\pi i\tau) \,,
&&\quad
h^{(1)}_n(q^{-1}) \hspace{-0.13in}&& \sim  
\wh\Phi^{(\s_1)}_n(-2\pi i\tau) + \wh\Phi^{(\s_2)}_n(-2\pi i\tau) \,,\\
h^{(2)}_n(q) & \sim \;-\frac{1}{6\tau}\wh\Phi^{(\s_1)}_n(2\pi i\tau) 
+\frac{1}{6\tau} \wh\Phi^{(\s_2)}_n(2\pi i\tau) \,,
&&\quad
h^{(2)}_n(q^{-1}) \hspace{-0.13in}&& \sim  -\frac{1}{6\tau}
\wh\Phi^{(\s_3)}_n(-2\pi i\tau) \,.
\end{aligned}
\end{tiny}
\ee
As in the case of the $4_1$ knot, each asympotic statement above
involves linear combinations of asymptotic series of the
\emph{same} exponential growth rate. 

Taking the vertical asymptotics of the quadratic relation~\eqref{52quad}
of the colored holomorphic blocks gives the following quadratic relation of the
three asymptotic series (given in~\cite[Sec.3.3]{GZ:kashaev} for $n=0$)
\be
\label{52Phiquadrel}
\widehat{\Phi}^{(\s_1)}_{n}(\hbar)\widehat{\Phi}^{(\s_1)}_{n}(-\hbar)
+\widehat{\Phi}^{(\s_2)}_{n}(\hbar)\widehat{\Phi}^{(\s_2)}_{n}(-\hbar)
+\widehat{\Phi}^{(\s_3)}_{n}(\hbar)\widehat{\Phi}^{(\s_3)}_{n}(-\hbar)=0\,,
\ee
valid for all integers $n$.

Equations~\eqref{52Irot} and~\eqref{52hvasy} imply that 
\be
\begin{tiny}
\begin{aligned}
\label{52.irot.asymp}
\Irot(n,n')(q)
\sim
-\widehat{\Phi}^{(\s_{1})}_{n}(2\pi i\tau)
\widehat{\Phi}^{(\s_{1})}_{n'}(-2\pi i\tau)
+\widehat{\Phi}^{(\s_{2})}_{n}(2\pi i\tau)
\widehat{\Phi}^{(\s_{3})}_{n'}(-2\pi i\tau)
+\widehat{\Phi}^{(\s_{3})}_{n}(2\pi i\tau)
\widehat{\Phi}^{(\s_{2})}_{n'}(-2\pi i\tau) \,.
\end{aligned}
\end{tiny}
\ee
verifying Conjecture~\eqref{conj.2} for the $5_2$ knot.



\bibliographystyle{plain}
\bibliography{biblio}
\end{document}